%% file: main.tex
\documentclass[11pt]{article}
\pdfoutput=1

\usepackage{amsfonts,amsmath,amssymb,amsbsy,amstext,amsthm, accents, enumerate, float, fullpage, verbatim, mathtools, url, microtype,eucal, bm,kpfonts,mathbbol, parskip, cases, todonotes}
\usepackage[shortlabels]{enumitem}
\usepackage[colorlinks=true,hyperindex, linkcolor=blue, pagebackref=false, citecolor=cyan]{hyperref}
\usepackage[alphabetic,lite,backrefs]{amsrefs}

\usepackage{lscape}
\usepackage{cleveref}

\usepackage{graphicx,pstricks}
\usepackage{graphics}
\usepackage{subcaption}
\usepackage{import}
\graphicspath{ {images/} }
\usepackage{subfiles}
\usepackage{float, pdfpages}
\usepackage{amsmath}

\allowdisplaybreaks[3]

\input{Figure/tikzstyle}

\usepackage[all,cmtip]{xy}
\usepackage{tikz}
\usetikzlibrary{arrows, decorations.markings, cd}
\tikzstyle arrowstyle=[scale=1]
\tikzstyle directed=[postaction={decorate,
decoration={markings,mark=at position .65 with {\arrow[arrowstyle]{stealth}}}}]
\usepackage{mathdots}
\usepackage{bbold} 
\usepackage{tikz-cd}   
\usepackage[shortlabels]{enumitem}
\usepackage{ytableau}
\usepackage[all]{nowidow}

\newcommand{\Z}{\mathbb{Z}}
\newcommand{\N}{\mathbb{N}}
\newcommand{\R}{\mathbb{R}}
\input macro.tex
\title{The Monoid Representation of Upho Posets and Total Positivity}
\author{Ziyao Fu\thanks{Peking University, Beijing \href{mailto:fzy1003@stu.pku.edu.cn}{fzy1003@stu.pku.edu.cn}}, Yulin Peng\thanks{Tsinghua University, Beijing \href{mailto:pyl24@mails.tsinghua.edu.cn}{pyl24@mails.tsinghua.edu.cn}}, and Yuchong Zhang\thanks{University of Michigan, Ann Arbor, MI 48104 \href{mailto:zongxun@umich.edu}{zongxun@umich.edu}}}     
\date{}
\begin{document}
\maketitle
\begin{abstract}
We show that all totally positive formal power series with integer coefficients and constant term $1$ are precisely the rank-generating functions of Schur-positive upho posets, thereby resolving the main conjecture proposed by Gao, Guo, Seetharaman, and Seidel. To achieve this, we construct a bijection between finitary colored upho posets and atomic, left-cancellative, invertible-free monoids, which restricts to a correspondence between $\mathbb{N}$-graded colored upho posets and left-cancellative homogeneous monoids. Furthermore, we introduce semi-upho posets and develop a convolution operation on colored upho posets with colored semi-upho posets within this monoid-theoretic framework.
\end{abstract}


\section{Introduction}\label{sec:intro}

A poset is called \emph{upper homogeneous}, abbreviated as \emph{upho}, if each principal order filter is isomorphic to the poset itself. Upho posets were introduced by Richard Stanley during his research on the enumeration properties of Stern's triangle and its poset \cites{Stanleyuphoslide, MR4752181}.

The study of Stern's poset, particularly regarding enumeration problems, has been a focal point of numerous research efforts \cites{MR4463221,MR4546049}. Upho posets, as a generalization of Stern's poset, preserve the attribute of self-similarity, and hence exhibit many intriguing structural and enumeration properties. For example, in \cite{MR4318812}, Gao et al. give a concise characterization of $\N$-graded planar upho posets using their rank-generating functions; in \cite{MR4432972}, Hopkins proves that the characteristic generating functions of upho posets are the inverse of their rank-generating functions. Moreover, the applications of upho posets span several domains, including lattice theory \cites{MR4432972, Hopkinslattice, HopkinsLewis24}, commutative algebra \cite[Conjecture~1.1]{MR4318812}, Coxeter groups \cite[Corollaries~4.10~and~4.15]{Hopkinslattice}, and finite geometry \cite[Theorem~1.3]{Hopkinslattice}.

For simplicity, we say a formal power series is an \emph{upho function} if it is the rank-generating function of some upho poset. A fundamental problem is:
\begin{Question}
    Is there a criterion to determine whether a formal power series is an upho function?
\end{Question}

In \cite{MR4318812}, Gao et al. prove that there are uncountably many different upho functions, and a straightforward corollary is that almost all upho functions are not rational functions. Hence, the complete characterization of upho functions is anticipated to be challenging. 

In this paper, we characterize the rank-generating functions of all \emph{Schur-positive upho posets}, which are upho posets with a Schur-positive Ehrenborg quasi-symmetric function, thereby resolving the main conjecture proposed by Gao et al. in \cite{MR4318812}.

\begin{theorem}[{\cite[Conjecture~3.3]{MR4318812}}]\label{conj:Gao1}
    A formal power series $f(x)\in 1+x\Z_{\ge 0}[[x]]$ is the rank-generating function of a Schur-positive upho poset if and only if $f(x)$ is totally positive.
\end{theorem}

In \Cref{sec:3}, we introduce the notion of \emph{colored upho posets} and establish a fundamental correspondence, detailed later, which serves as the primary tool for proving all results in this paper.

\begin{theorem}\label{thm: intro corresp} There exists a canonical bijection between finitary colored upho posets and atomic, left-cancellative, invertible-free monoids. Moreover, this bijection maps $\N$-graded colored upho posets to left-cancellative homogeneous monoids and maps finite-type $\N$-graded colored upho posets to finitely generated left-cancellative homogeneous monoids. \end{theorem}

This correspondence is crucial for the study of both upho posets and left-cancellative monoids. On one hand, it enables concrete calculations on upho posets using monoids from an algebraic perspective. On the other hand, it provides combinatorial insight into problems such as enumerating elements in left-cancellative monoids. We use \emph{regular upho posets} to denote upho posets that admit a compatible coloring (in fact, we conjecture that all finitary upho posets admit such a coloring). Furthermore, we refer to their rank-generating functions as \emph{regular upho functions}.

In \Cref{sec:4}, we explore the properties of \emph{semi-upho posets}, a generalization of upho posets that exhibit partial self-similarity. We build a framework to characterize all \emph{regular semi-upho functions}, and prove that all log-concave formal power series are regular semi-upho functions.

\begin{theorem}\label{thm: log-concave rgf of semi-upho}
    Any log-concave formal power series $g(x)\in 1+x\Z_{\ge 0}[[x]]$ is a regular semi-upho function.
\end{theorem}

In \Cref{sec:5}, we define the \emph{$\bx$-convolution} of a colored upho poset with a semi-upho spanning tree of a colored semi-upho poset, and shows that the $\bx$-convolution remains a colored upho poset using calculations based on monoids. A direct corollary is that multiplying a regular upho function by a regular semi-upho function yields another regular upho function.

\begin{theorem}\label{thm:upho times semi-upho}
        Let $f(x)\in 1+x\Z_{\ge 0}[[x]]$ be a regular upho function, and $g(x)\in 1+x\Z_{\ge 0}[[x]]$ be a regular semi-upho function, then $f(x)g(x)$ is a regular upho function.
\end{theorem}

This result serves as the main tool in the proof of \Cref{conj:Gao1}.

In \Cref{sec:6}, we explore the relationship between upho functions and totally positive power series.

\begin{definition}[{\cite{MR0230102}}]\label{def:TP}
    A formal power series $\sum_{n=0}^{\infty}a_n x^n$ is \emph{totally positive} if all finite minors of the infinite Toeplitz matrix
\[\begin{pmatrix}a_{0} & 0&0&0&\cdots \\a_{1} & a_{0} &0&0&\cdots \\a_{2} & a_{1} &a_{0}&0&\cdots \\\vdots&\vdots&\vdots&\vdots&\ddots\end{pmatrix}
\]
    are nonnegative.
\end{definition} 

Our working definition of totally positive formal power series follows from \Cref{thm:working def of TP}.

\begin{theorem}[{\cites{MR41897,MR11720}}]\label{thm:working def of TP}
    A formal power series $f(x)\in 1+x\Z_{\ge 0}[[x]]$ is totally positive if and only if $f(x)$ can be expressed in the form $\frac{g(x)}{h(x)}$, where $g(x), h(x) \in 1 + x \mathbb{Z}[x]$ are polynomials with all complex roots of $g(x)$ real and negative, and all complex roots of $h(x)$ real and positive.
\end{theorem}

By employing \Cref{thm: intro corresp}, \Cref{thm: log-concave rgf of semi-upho}, \Cref{thm:upho times semi-upho}, and resolving an exceptional case by a miracle in linear algebra, we prove the following theorem.

\begin{theorem}\label{thm:main}
    Any totally positive formal power series $f(x)\in 1+x\mathbb{Z}_{\ge 0}[[x]]$ is a regular upho function. 
\end{theorem}

And finally, we show that \Cref{conj:Gao1} is a corollary of \Cref{thm:main}.

The paper is structured as follows: \Cref{sec:2} provides necessary background on upho posets and introduces the notion of semi-upho posets. \Cref{sec:3} through \Cref{sec:6} are dedicated to the exposition and proofs of the aforementioned results. In \Cref{sec:7}, we propose further conjectures on upho posets and related structures and build a framework for characterizing all upho functions.

\section{Preliminaries}\label{sec:2}
In this paper, we employ the standard terminology used in order theory. For a more detailed exposition, readers are referred to \cite[Chapter~3]{MR1442260}.
\subsection{Upho Posets}
In this subsection, we introduce several definitions related to upho posets, accompanied by illustrative examples.

\begin{definition}
A poset $P$ is \emph{upper homogeneous}, abbreviated as \emph{upho}, if for any $s\in P$ we have $V_{P,s}\cong P$, where $V_{P,s}\coloneqq\{p\in P\mid s\le_P p\}$ is the \emph{principal order filter} generated by $s$.
\end{definition}
Note that each principal order filter has a unique minimal element, so every upho poset $P$ has a unique minimal element, denoted $\hat{0}_P$. We abbreviate $V_{P,s}$ to $V_s$ and $\hat{0}_P$ to $\hat{0}$ when the poset $P$ referred to is clear.

A poset $P$ is said to be \emph{$\mathbb N$-graded} if $P$ can be written as a disjoint union $P=P_0\sqcup P_1\sqcup P_2\sqcup\dots$ such that every saturated chain has the form $p_0\lessdot_P p_1\lessdot_P p_2\lessdot_P \cdots$, where $p_i\in P_i$ for all $i\in \N$. The \emph{rank function} $\rho:P\rightarrow \N$ of $P$ is defined by $\rho(p)=i$ for all $p\in P_i$. We refer to $P_i$ as the \emph{$i$-th layer} of $P$. 

An $\N$-graded poset $P$ is said to be of \emph{finite type} if $\left|P_i\right|$ is finite for all $i\in \N$. The \emph{rank-generating function} of a finite type $\N$-graded poset $P$ is defined to be $F_P(x)\coloneqq\sum_{k=0}^{\infty}\left|P_k\right| x^k$.

\begin{definition}
A formal power series $f(x)\in 1+x\Z[[x]]$ is an \emph{upho function} if it is the rank-generating function of a finite type $\N$-graded upho poset.    
\end{definition}

\begin{definition}[{\cite{MR1383883} }]
    The \emph{Ehrenborg quasi-symmetric function} of a finite type $\mathbb{N}$-graded poset $P$ is defined to be $E_P\coloneqq\sum_{n \geq 0} E_{P, n}$, where $E_{P,0}=1$, and
\[
E_{P, n}\left(x_1, x_2, \cdots \right)\coloneqq\sum_{\substack{\hat{0} = t_0 \leq_P t_1 \leq_P \cdots \leq_P t_{k-1}<_P t_k \\ \rho\left(t_k\right)=n}} x_1^{\rho(t_1)-\rho(t_0)} x_2^{\rho(t_2)-\rho(t_1)} \cdots x_k^{\rho(t_k)-\rho(t_{k-1})},\ n \ge 1.
\]
\end{definition}

\begin{definition}[\cite{MR4318812}]
    An finite type $\N$-graded upho poset is a \emph{Schur-positive upho poset} if its Ehrenborg quasi-symmetric function is a Schur-positive symmetric function.
\end{definition}

We state some fundamental properties of upho functions.

\begin{lemma}[{\cite[Lemma~2.2]{MR4318812}}]\label{lemma: Ep of upho symmetric}
    For a finite type $\N$-graded upho poset $P$, $E_P(x_1,x_2,\cdots )=F_P(x_1)F_P(x_2)\cdots $, hence $E_p$ is a symmetric function.
\end{lemma}

Recall that the \emph{product} of two posets $P$ and $Q$, denoted $P\times Q$, is a poset whose elements are given by the Cartesian product of $P$ and $Q$, with the order defined by $(a,b)\leq (c,d)$ iff $a\leq_{P}c$ and $b\leq_{Q}d$.

\begin{lemma}[{\cite[Lemma~2.3]{MR4318812}}]\label{lemma:produpho}
    Let $P$ and $Q$ be finite type $\N$-graded upho posets. Then $P\times Q$ is a finite type $\N$-graded upho poset. Furthermore, $F_{P\times Q}=F_{P} F_{Q}$, and $E_{P\times Q}=E_{P} E_{Q}$.
\end{lemma}

We define the following two finiteness conditions. It is worth noting that the terminology is not standard, and ``finitary" here does not align with its use in ``finitary distributive lattice."

\begin{definition}
A poset $P$ is \emph{Noetherian} if its principal order filters satisfy \emph{ascending chain condition}, that is, for any element $s\in P$, there is no infinite strictly ascending chain $V_s \subseteq V_{s_1}\subseteq V_{s_2}\subseteq \cdots$ (or equivalently, $s>_Ps_1>_Ps_2>_P\cdots$).
\end{definition}

\begin{definition}
    The \emph{height} of an element $s$ in a poset $P$ is the minimal length of saturated chains in $P$ with $s$ as its maximum. A poset $P$ is \emph{finitary} if every element in $P$ has finite height.
\end{definition}

All $\N$-graded posets are finitary, and all finitary posets are Noetherian. Moreover, given a Noetherian upho poset $P$, a finitary upho poset $P'\subseteq P$ can be obtained by selecting all elements of finite height in $P$, with $P'$ inheriting the order of $P$.

\begin{notation}
    In a poset $P$, we denote the set of \emph{edges} of $P$ by $\mathcal{E}_P\coloneqq \{(r, s)\mid r,s \in P, r\lessdot s\}$, which corresponds to the edges in the Hasse diagram of $P$ if $P$ is finitary. If a poset $P$ has a unique minimum $\hat{0}$, we denote the set of \emph{atoms} of $P$ by $\mathcal{A}_P\coloneqq \{s\in P \mid \hat{0}\lessdot s\}$.
\end{notation} 
It can be easily verified that in a Noetherian poset, each saturated chain includes exactly one atom. However, both $\mathcal{A}_P$ and $\mathcal{E}_P$ can be empty if $P$ is non-Noetherian. See \Cref{eg:R}.

We list below some examples of upho posets and upho functions.

\begin{example}\label{eg:R}
     Nonnegative real numbers $\R_{\ge 0}$ with usual order is a non-Noetherian upho poset, and there are no atoms in $\R_{\ge 0}$. Let $\R_{\ge 0}\times \N$ be the poset product of $\R_{\ge 0}$ and $\N$, both with usual order. Then $\R_{\ge 0}\times \N$ is also a non-Noetherian upho poset, and $(0,1)$ is its only atom.
\end{example}

\begin{example}\label{eg:N*N lexi}
    The poset $\N \times \N$ with lexicographic order forms a Noetherian upho poset which is not finitary. For any $(m,n) \in \N \times \N$, the isomorphism $\tau: V_{(m,n)}\cong \N \times \N$ is given by $\tau(m,n+s)=(0,s), \tau(m+t, s)=(t,s)$ for all $s\in \N, t\in \N_{>0}$.
\end{example}

\begin{example}\label{eg:not N graded}
    The upho poset $P_M$ is defined as follows. The elements in $P_M$ are those in the monoid $M = \langle x_1, x_2 \mid x_1^3 = x_2 x_1 \rangle$. The partial order of $P_M$ is defined by left-divisibility in $M$, that is, $a \le_{P_M} b$ if and only if there exists $c \in M$ such that $ac = b$ in $M$. 
\end{example}
It can be shown that $P_M$ is a finitary upho poset but not an $\N$-graded one.

\begin{example}\label{eg:not finite type}
The upho posets $P_{M_1}$ and $P_{M_2}$ are defined as follows. The upho poset $P_{M_1}$ consists of elements in the free monoid $M_1$ generated by $[0,1] \subseteq \R$, with a partial order defined by left-divisibility. Similarly, $P_{M_2}$ consists of elements in the free commutative monoid $M_2$ generated by $[0,1] \subseteq \R$, with its partial order also defined by left-divisibility.
\end{example}

Both $P_{M_1}$ and $P_{M_2}$ are $\N$-graded upho posets that are not of finite type. The $d$-th layer of $P_{M_1}$ and $P_{M_2}$ can be viewed as a $d$-dimensional cube and a $d$-dimensional simplex, respectively.

\begin{example}\label{example:easyexamples}
    \Cref{fig:easyeg} shows the Hasse diagrams of several $\N$-graded upho posets of finite type: $\mathbb{N}$ (with usual order), the full binary tree, Stern's poset, and the bowtie poset, from left to right. Their rank-generating functions are $\frac{1}{1-x}$, $\frac{1}{1-2x}$, $\frac{1}{(1-x)(1-2x)}$, and $\frac{1+x}{1-x}$, respectively.
    \end{example}
    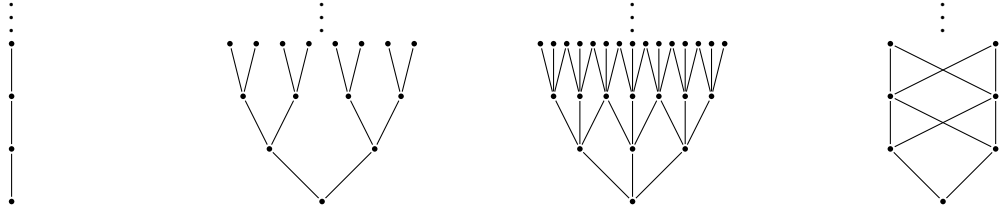
\begin{figure}[H]
\centering
    \begin{minipage}{.25\textwidth}
    \centering
    \input{Figure/1-x}
    \end{minipage}%
    \begin{minipage}{.25\textwidth}
    \centering
    \input{Figure/nraytree}
    \end{minipage}%
    \begin{minipage}{.25\textwidth}
    \centering
    \input{Figure/stern}
    \end{minipage}%
    \begin{minipage}{.25\textwidth}
    \centering
    \input{Figure/bowtie}
    \end{minipage}%
    \caption{The Hasse diagrams of \Cref{example:easyexamples}.}
    \label{fig:easyeg}
\end{figure}

\begin{example}\label{example:feiposet}
Fei's poset $\mathcal{F}$ is the upho poset that satisfies the generating rules depicted in the left two components of \Cref{fig:fei}. This poset is a finite type $\N$-graded upho poset, with its rank-generating function given by $F_{\mathcal{F}}(x)=\frac{1-x}{1-2x-x^2}$. The first few layers of the Hasse diagram of Fei's poset are shown in the rightmost component of \Cref{fig:fei}.
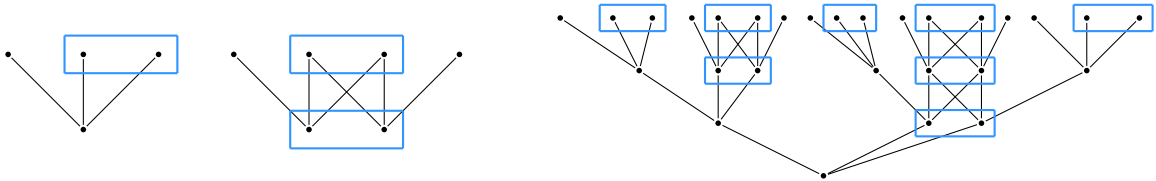
\begin{figure}[H]
    \begin{minipage}{.45\textwidth}
    \centering
    \input{Figure/feigenerate}
    \end{minipage}%
    \begin{minipage}{.55\textwidth}
    \centering
    \input{Figure/fei}
    \end{minipage}%
\caption{The generating rules and the first several layers of Fei's poset.}
\label{fig:fei}
\end{figure}
\end{example}
On one hand, \Cref{example:feiposet} shows that upho functions are not necessarily log-concave (see \Cref{def:log-concave}), since $|\mathcal{F}_1|\cdot|\mathcal{F}_3|=51>49=|\mathcal{F}_2|^2$. On the other hand, it illustrates potential connections between upper homogeneity and linear recurrence, inspiring our proof of \Cref{prop:type II}.

\subsection{Semi-Upho Posets}
In this subsection, we introduce \emph{semi-upho posets}, a generalization of upho posets exhibiting partial self-similarity. The primary motivation for defining such posets is to identify suitable formal power series $g(x)$ such that $f(x)g(x)$ is an upho function for any upho function $f(x)$. As shown in \Cref{lemma:produpho}, upho functions serve as appropriate choices for $g(x)$. Furthermore, in \Cref{thm:upho times semi-upho}, we extend this result to the rank-generating functions of semi-upho posets, with a mild restriction termed ``regular."

\begin{definition}
Given posets $P'$ and $P$ with unique minima $\hat{0}_{P'}$ and $\hat{0}_P$ respectively, an injective mapping $\eta : P' \hookrightarrow P$ is an \emph{induced saturated order embedding}, abbreviated as \emph{isoembedding}, if $\eta(\hat{0}_{P'})=\hat{0}_P$, and furthermore, for any chain $C$ in $P'$ with given maximum $a$ and minimum $b$, $C$ is a saturated chain with given maximum $a$ and minimum $b$ if and only if $\eta(C)$ is a saturated chain with maximum $\eta(a)$ and minimum $\eta(b)$.     
\end{definition}

For finitary posets, an isoembedding is an embedding that preserve the covering relation $\lessdot$.

\begin{definition}
A poset $S$ is a \emph{semi-upho poset} if, for any $s\in S$, there exists an isoembedding $V_s \hookrightarrow S$. A formal power series $f(x)\in 1+x\Z[[x]]$ is a \emph{semi-upho function} if it is the rank-generating function of a finite type $\N$-graded semi-upho poset.    
\end{definition}

Upho posets form a subclass of semi-upho posets as isomorphisms are isoembeddings. Moreover, intuitively, a semi-upho poset can be thought of as an upho poset with some parts cut off. 

\begin{definition}
A finitary semi-upho poset $S$ is \emph{tree-like} if for any two elements $s, t \in S$, there does not exist an element $u \in S$ such that $s <_S u$ and $t <_S u$.
\end{definition}

The Hasse diagram of a tree-like semi-upho poset is a tree. 

\begin{example}\label{eg:tree-like sem-upho}
    \Cref{fig:treesemi} is an example of a tree-like semi-upho posets. The red part is a principal order filter that can be isoembedded into the poset itself.

\begin{figure}[H]
    \centering
    \input{Figure/semiupho}
    \caption{A tree-like semi-upho poset.} 
    \label{fig:treesemi}
\end{figure}
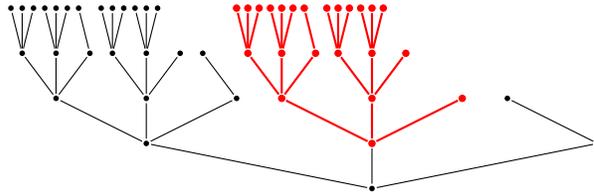
\end{example}

\subsection{Monoids}
In this paper, our main focus is on the upho posets derived from left-cancellative homogeneous monoids. Accordingly, we introduce some definitions and properties related to these monoids. For a more detailed exposition, readers are referred to \cite{MR1455373}. It is worth noting that certain terms in this paper, such as ``homogeneous" and ``invertible-free," are not yet standard terminology. 

A monoid $M$ is a set equipped with an associative binary operation, called \emph{multiplication}, and an \emph{identity element} $e$ such that $ex=x=xe$ for any $x \in M$. 

\begin{definition}
     A \emph{free monoid} $\mathbf{F}(\mathbf{X})$ generated by a set $\mathbf{X}$ is defined as follows:
\begin{itemize}
    \item Elements in $\mathbf{X}$ are called \emph{letters}. A finite sequence of letters $W=a_1a_2\cdots a_n$ is called a \emph{word}, with \emph{length} denoted by $\ell(W)\coloneqq n$. There exists a unique word of length $0$, denoted by $e$. 
    \item $\mathbf{F}(\mathbf{X})$ is the set of all words formed from letters in $\mathbf{X}$, with multiplication of two words $W_1$ and $W_2$ defined by their concatenation.
\end{itemize}
\end{definition}   

Let $\mathbf{F}_k(\mathbf{X})$ denote the set of all words of length $k$ in $\mathbf{F}(\mathbf{X})$. To facilitate working with elements in monoids, we define the \emph{$k$-lexicographic order} $<_k$ on $\mathbf{F}_k(\mathbf{X})$ as follows.

\begin{definition}\label{def: k-lexi order of F_k}
Given an ordered set $(\mathbf{X}, <)$, the \emph{$k$-lexicographic order} $<_k$ on $\mathbf{F}_k(\mathbf{X})$ is the total order defined by: for two distinct words $W_1 = a_1 a_2 \ldots a_k$ and $W_2 = b_1 b_2 \ldots b_k$ in $\mathbf{F}_k(\mathbf{X})$, we have $W_1 <_k W_2$ if and only if $a_i < b_i$ at the first position $i$ (counting from the left) where $W_1$ and $W_2$ differ.
\end{definition}

     A \emph{relation} on a set $S$ is a subset $R \subseteq S \times S$. It is an \emph{equivalence} if it is reflexive $(x, x) \in R$, symmetric $(x, y) \in R \Rightarrow (y, x) \in R$, and transitive $(x, y), (y, z) \in R \Rightarrow (x, z) \in R$. A \emph{congruence} on a monoid $M$ is an equivalence $R$ defined on $M$ such that $(xay, xby) \in R$ for any $(a, b) \in R$ and $x, y \in M$.

Given a relation $R$ on a monoid $M$, we denote by $R^c$ the smallest congruence containing $R$. Given a monoid $M$, if $M \cong \mathbf{F}(\mathbf{X}_M)/{R_M^c}$ for a relation $R_M$ on a set $\mathbf{X}_M$, then $M = \langle \mathbf{X}_M \mid R_M \rangle$ is called a \emph{presentation} of $M$. We call $\mathbf{X}_M$ the \emph{generating set} and elements of $R_M$ the \emph{defining relations} of $M$. If $\mathbf{X}_M = \{s_1, \dots, s_n\}$ and $R_M = \{(a_1, b_1), \dots, (a_m, b_m)\}$ are finite, then $M$ is called \emph{finitely presented} and written as $M = \langle s_1, \dots, s_n \mid a_1 = b_1, \dots, a_m = b_m \rangle$.

For any element $W \in M = \langle \mathbf{X}_M \mid R_M \rangle$, if $W = a_1 a_2 \cdots a_n$ with each $a_i \in \mathbf{X}_M$, then $a_1 a_2 \cdots a_n$ is called a \emph{word representation} of $W$. Moreover, if two word representations $W_1$ and $W_2$ correspond to the same element $W$ in $M$, we write $W_1 = W_2 = W$. If, in addition, these word representations are identical in $\mathbf{F}(\mathbf{X}_M)$, meaning they use the same letters in the same order, we write $W_1 \equiv W_2$.

A \emph{zero element} of $M$, denoted $0 \in M$, satisfies $0x = x0 =0$ for any $x \in M$. 

An element $x \in M$ is \emph{left-invertible} if there exists $y \in M$ such that $yx = e$, and \emph{right-invertible} if there exists $y \in M$ such that $xy = e$. An element $x \in M$ is \emph{invertible} if it is both left- and right-invertible. 

\begin{definition}
    A monoid $M$ is \emph{invertible-free} if it has no left-invertible or right-invertible elements other than $e$; equivalently, $xy = e$ implies $x = y = e$ for any $x, y \in M$. 
\end{definition}

An element of $M$ is \emph{irreducible} if it is not invertible and is not the product of two non-invertible elements. The set of all irreducible elements in $M$ is denoted $\mathcal{I}_M$. A monoid is \emph{atomic} if all its elements can be written as a finite product of irreducible elements.

\begin{definition}
    An atomic monoid $M$ is \emph{homogeneous} if, for any $W \in M$ with $W = a_1 a_2 \cdots a_n = b_1 b_2 \cdots b_m$, where $a_i, b_j \in \mathcal{I}_M$, we have $n = m$. We define the \emph{length} of $W$, denoted $\ell(W)$, to be $n$ and let $\mathbf{W}^M_n$ be the set of all distinct elements of length $n$. It can be verified that homogeneous monoids are invertible-free.
\end{definition}

Given a homogeneous monoid $M$, we define the \emph{$k$-lexicographic order} $<^k$ on $\mathbf{W}^M_k$ as follows.

\begin{definition}\label{def: k-lexi order of W^M_n}
Let $M$ be a homogeneous monoid with a total order defined on $\mathcal{I}_M$. The \emph{$k$-lexicographic order} $<^k$ on $\mathbf{W}^M_k$ is defined by: for two elements $W_1, W_2 \in \mathbf{W}^M_k$, we have $W_1 \le^k W_2$ if and only if the minimal word representations of $W_1$ and $W_2$ under the $k$-lexicographic order $<_k$ on $\mathbf{F}_k(\mathcal{I}_M)$ (as introduced in \Cref{def: k-lexi order of F_k}), denoted $W'_1$ and $W'_2$, satisfy $W'_1 \le_k W'_2$.
\end{definition}

For $a, b \in M$, we say $a$ is \emph{left-divisible} by $b$ if $a = bc$ for some $c \in M$. A monoid $M$ is \emph{left-cancellative} if for any $a, x, y \in M$, $ax = ay$ implies $x = y$.

\begin{notation}
    For simplicity, we abbreviate left-cancellative invertible-free monoids as \emph{LCIF monoids}, and abbreviate left-cancellative homogeneous monoids as \emph{LCH monoids}.
\end{notation}

\begin{remark}\label{remark: LCIF not atomic}
    An LCIF monoid is not necessarily atomic. A counterexample is $M=\langle \mathbf{X}\mid R\rangle$, where $\mathbf{X}=\{a, b_i\mid i\in \N\}$, and $R=\{ b_i=ab_{i+1}\mid i\in \N\}$.
\end{remark}

Our main technique for characterizing equalities in a monoid is the following property.
      \begin{prop}\cite{MR1455373}\label{prop:monoidtrans}
        Let $R$ be an equivalence on a set $S$, and $c,d$ be two elements in S. We say that $c\rightarrow d$ is a \emph{elementary $R$-transition} if there exists $(a,b)\in R$ and $x,y\in S$ such that $c=xay, d=xby$. Then $(x,y)\in R^c$ if and only if there exists a finite sequence \[x\equiv z_1\rightarrow z_2\rightarrow \cdots \rightarrow z_n\equiv y\] of elementary $R$-transition connecting $x$ and $y$.
     \end{prop}

\subsection{Power Series}
In the second half of this paper, we focus on determining which formal power series are upho functions. Alongside totally positive formal power series (\Cref{def:TP}, \Cref{thm:working def of TP}), we recall some other definitions and fundamental results related to formal power series.
In this paper, a polynomial is regarded as a formal power series with finitely many nonzero terms.

Recall that a symmetric function $f$ is said to be \emph{Schur-positive} if, when expanded in the linear basis of Schur polynomials $f=\sum_\mu a_\mu s_\mu$, all coefficients $a_\mu$ are positive. We have the following property relating Schur-positivity and total positivity, which can be proved using Jacobi-Trudi identity.

\begin{prop}\cite{MR3618143}\label{prop:Schur positive}
     A formal power series $f(x)\in 1+x\Z_{\ge 0}[[x]]$ is totally positive if and only if $\prod_{i\ge 1} f(x_i)$ is Schur positive.
\end{prop}

In this paper, a log-concave formal power series is required to be nonnegative, and has no internal zeros.

\begin{definition}\label{def:log-concave}
A \emph{log-concave formal power series} is a formal power series
    \[
    a_0+a_1x+a_2x^2+\cdots+a_nx^n+\cdots
    \]
    satisfying $a_i\ge 0$, $a_ia_{i+2}\le a_{i+1}^2$ for any integer $i$, and moreover, if $a_j=0$, then $a_k=0$ for all $k\ge j$.
\end{definition}

A fact we use in this paper is that multiplication preserves the log-concavity of formal power series.
    
\begin{lemma}\cite{MR1110850}\label{lemma:logconcavetimeslogconcave}
    If $A(x), B(x)$ are log-concave formal power series, then so is $A(x)B(x)$.
\end{lemma}

\section{Monoid Representations}\label{sec:3}
In this section, we establish a bijection between finitary \emph{colored upho posets} and atomic LCIF monoids, and associate the bijection to upho posets.

\subsection{Colored Upho Posets}\label{subsec:colored upho def}
The motivation for introducing \emph{colored upho posets} is to designate a unique isomorphism between each given principal order filter and the upho poset itself.

\begin{definition}
    A colored upho poset $\tilde{P}$ consists of the data $(P, \col_P)$: the poset $P$ is an upho poset, and the color mapping $\col_P: \mathcal{E}_P\to \mathcal{A}_P$ satisfies the following conditions:
    \begin{itemize}
        \item For any $t\in \mathcal{A}_P$, $\col_P(\hat{0}, t)=t$;
        \item For any $s\in P$, there exists an isomorphism $\phi_s: V_{s}\xrightarrow{\sim} P$ such that $\col_P(u,v)=\col_P(\phi_s(u), \phi_s(v))$ for all $(u,v)\in \mathcal{E}_{V_s}$.
    \end{itemize}
\end{definition}
It can be shown that such $\phi_s$ is unique for a given $s\in P$.

Finitary colored upho posets can be conceptualized as structures in which each edge extending upward from the same vertex in their Hasse diagrams is assigned a distinct color. Moreover, there exists a unique isomorphism between each principal order filter and the poset itself, which maps edges to identically colored ones.

Similarly, we introduce colored semi-upho posets.

\begin{definition}
        A colored semi-upho poset $\tilde{S}$ consists of the data $(S, \col_S)$: the poset $S$ is a semi-upho poset, and the color mapping $\col_S: \mathcal{E}_S\to \mathcal{A}_S$ satisfies the following conditions:
    \begin{itemize}
        \item For any $t\in \mathcal{A}_S$, $\col_S(\hat{0}, t)=t$;
        \item For any $s\in S$, there exists an isoembedding $\psi_s: V_{s}\hookrightarrow S$ such that $\col_S(u,v)=\col_S(\psi_s(u), \psi_s(v))$ for all $(u,v)\in \mathcal{E}_{V_s}$. 
    \end{itemize}
\end{definition}

It can be shown that such $\psi_s$ is also unique for a given $s\in S$.

\subsection{Correspondence with Monoids}\label{subsec:corresp}
In this subsection, we build the bijection between finitary colored upho posets and atomic LCIF monoids explicitly. 

    We define a mapping $\mathcal{M}$ which maps finitary colored upho posets to LCIF monoids by the following rule.
    For a given finitary colored upho poset $\tilde{P}=(P,\col_P)$, the elements of $\mathcal{M}(\tilde{P})$ are the elements of $P$. For any $s,t \in P$, we define the multiplication $st$ by $st \coloneqq \phi_s^{-1}(t)$. 
    
    \begin{lemma}\label{lemma:poset to LCIF monoid}
        The monoid $\mathcal{M}(\tilde{P})$ is a well-defined atomic LCIF monoid, and  $\mathcal{A_P}=\mathcal{I}_{\mathcal{M}(\tilde{P})}$.
    \end{lemma}
    \begin{proof}
        In the colored upho poset $P$, since both $\phi_{s_1s_2}^{-1}$ and $\phi_{s_1}^{-1}\circ \phi_{s_2}^{-1}$ induces an isomorphism $P\xrightarrow{\sim}V_{s_1s_2}$, the uniqueness of such isomorphism in colored upho poset implies that $\phi_{s_1s_2}^{-1}=\phi_{s_1}^{-1}\circ \phi_{s_2}^{-1}$. Therefore, the monoid $\mathcal{M}(\tilde{P})$ is well-defined since the identity element $e\in \mathcal{M}(\tilde{P})$ corresponds to $\hat{0}\in P$ and \[(s_1s_2)t=\phi_{s_1s_2}^{-1}(t)=\phi_{s_1}^{-1}(\phi_{s_2}^{-1}(t))=s_1(s_2t)\]
        for any $s_1, s_2, t\in \mathcal{M}(\tilde{P})$.
        
        Suppose $st_1=st_2$ in $\mathcal{M}(\tilde{P})$, then $\phi_s^{-1}(t_1)=\phi_s^{-1}(t_2)$ in $P$. Since $\phi_s^{-1}$ is an isomorphism, it follows that $t_1=t_2$, and hence $\mathcal{M}(\tilde{P})$ is left-cancellative. 
        
        Suppose $st=e$ in $\mathcal{M}(\tilde{P})$, then $\phi_s^{-1}(t)=e$, which implies $s=e$. Therefore, $st=e$ implies $s=t=e$, and hence $\mathcal{M}(\tilde{P})$ is invertible-free. 
        
        The irreducible elements of $\mathcal{M}(\tilde{P})$ are exactly $\mathcal{A_P}$. Note that any nontrivial element $t\in \mathcal{M}(\tilde{P})$ lie in the order filter generated by some atom $s$. If $t\notin \mathcal{A_P}$, then $t=\phi_{s}^{-1}(t_1)$ for some $t_1\neq e$, so $t=st_1\notin \mathcal{I}_{\mathcal{M}(\tilde{P})}$. This implies $\mathcal{I}_{\mathcal{M}(\tilde{P})}\subseteq \mathcal{A_P}$, and the converse inclusion is clear by definition.
        
        Finally, since $P$ is finitary, any element in $\mathcal{M}(\tilde{P})$ can be written as a finite product of elements in $\mathcal{A_P}=\mathcal{I}_{\mathcal{M}(\tilde{P})}$, and therefore, $\mathcal{M}(\tilde{P})$ is atomic.
    \end{proof}

    Conversely, we define a mapping $\mathcal{\tilde{P}}=(\mathcal{P},\mathcal{C})$ which maps atomic LCIF monoids to finitary colored upho posets by the following rule.
    The elements of $\mathcal{P}(M)$ are the elements of $M$. The partial order $\le_\mathcal{P}(M)$ is defined by the left-divisibility in $M$, as is in \Cref{eg:not N graded}. Then we have for any $a, b \in M$, $a\lessdot_P b$ if and only if there exists $c\in \mathcal{I}_M$ such that $ac=b$. So the unique minimum in $\mathcal{P}(M)$ is $e$, and $\mathcal{A}_{\mathcal{P}(M)}=\mathcal{I}_M$. For any ordered pair $(a, b)\in\mathcal{E}_{\mathcal{P}(M)}$, define $\mathcal{C}(M)(a,b)$ to be the $c\in \mathcal{I}_M$ such that $ac=b$. Such $c$ is unique because $b=ac=ac'$ implies $c=c'$ by the left-cancellative property of $M$.

    \begin{lemma}[cf. {\cite[Lemma~5.1]{MR4318812}}]\label{lemma:LCIF monoid to posets}
        The colored upho poset $\mathcal{\tilde{P}}(M)=(\mathcal{P}(M),\mathcal{C}(M))$ is a well-defined finitary colored upho poset. 
    \end{lemma}
    \begin{proof}
    Given elements $a,b\in \mathcal{P}(M)$ such that $a \le_{\mathcal{P}(M)} b$ and $b\le_{\mathcal{P}(M)}a$, there exists $c,d\in \mathcal{P}(M)$ such that $a=bc$ and $b=ad$. Thus, $a=adc$. Since $M$ is left-cancellative and invertible-free, it follows that $e=dc$, $c=d=e$, and $a=b$. Therefore, $a \le_{\mathcal{P}(M)} b$ and $b\le_{\mathcal{P}(M)}a$ implies $a=b$. Furthermore, reflexivity and transitivity are straightforward to verify, so the partial order $\le_{\mathcal{P}(M)}$ is well-defined.
    
By left-divisibility, each element in $V_W \subseteq \mathcal{P}(M)$ can be written as $WX$ for some $X \in M$. The mapping $\phi_W : V_W \rightarrow \mathcal{P}(M)$ defined by $\phi_W(WX) = X$ is a homomorphism, as $X \le_{\mathcal{P}(M)} Y$ if and only if $WX \le_{\mathcal{P}(M)} WY$ by the left-cancellative property of $M$. Moreover, this mapping is bijective, thus an isomorphism, and hence $\mathcal{P}(M)$ is an upho poset.

For any $t \in \mathcal{A}_{\mathcal{P}(M)}$, $\mathcal{C}(M)(\hat{0}, t)$ is given by the unique $c \in \mathcal{I}_M$ such that $ec = t$, implying $c = t$. Thus, $\mathcal{C}(M)(\hat{0}, t) = t$ for all $t \in \mathcal{A}_{\mathcal{P}(M)}$. Moreover, fixing $W \in M$ and the corresponding isomorphism $\phi_W : V_W \rightarrow \mathcal{P}(M)$, we have, for any $(WU, WV) \in \mathcal{E}_{V_W}$, that $\mathcal{C}(M)(WU, WV)$ is defined by the unique $c \in \mathcal{I}_M$ such that $WUc = WV$. By left-cancellative property, this implies $Uc = V$, hence
\[
\mathcal{C}(M)(WU, WV) = c = \mathcal{C}(M)(U, V) = \mathcal{C}(M)(\phi_W(WU), \phi_W(WV)),
\]
as required.

Finally, $\mathcal{\tilde{P}}(M)$ is finitary because, since $M$ is atomic, any element in $M$ can be written as a finite product of elements in $\mathcal{I}_M = \mathcal{A}_{\mathcal{P}(M)}$, which induces a finite saturated chain.
    \end{proof}

Now we are ready to prove the main theorem of this section.

    \begin{proof}[Proof of \Cref{thm: intro corresp}]
        This follows from \Cref{lemma:poset to LCIF monoid} and \Cref{lemma:LCIF monoid to posets}, along with the fact that $\mathcal{\tilde{P}}(\mathcal{M}(\tilde{P}))\cong \tilde{P}$ and $\mathcal{M}(\mathcal{\tilde{P}}(M))\cong M$, both of which are straightforward to verify.
    \end{proof}

    More precisely, we have:
    
    \begin{theorem}\label{thm:corrsp for finitary upho}
        The mutually inverse mappings $\mathcal{M}$ and $\mathcal{\tilde{P}}$ give a bijection between finitary colored upho posets and atomic LCIF monoids.
    \end{theorem}

    Restrict to $\N$-graded colored upho posets, we have:

    \begin{cor}\label{cor:corrsp for N-graded upho}
        The mutually inverse mappings $\mathcal{M}$ and $\mathcal{\tilde{P}}$ give a bijection between $\N$-graded colored upho posets and LCH monoids. Moreover, finite type $\N$-graded colored upho posets correspond to finitely generated LCH monoids.
    \end{cor}

    Furthermore, the bijection can be generalized to semi-upho posets. 

    \begin{definition}
            An atomic \emph{LCIF $0$-monoid} is defined as an atomic LCIF monoid with an additional zero element $0$ and certain relations of the form $W_i = 0$, where each $W_i$ is an element of the original atomic LCIF monoid. We refer to these relations as \emph{$0$-defining relations}. Similarly, an \emph{LCH $0$-monoid} is an LCH monoid with an additional zero element $0$ and some $0$-defining relations, and a \emph{free $0$-monoid} is a free monoid with an additional zero element $0$ and some $0$-defining relations.
    \end{definition}

    \begin{notation}
        Given an LCH $0$-monoid $M$, we use $\mathbf{W}^M_k$ to denote the set of all distinct \textbf{nonzero} elements of length $k$ in $M$.
    \end{notation}
    
    \begin{remark}
        It is worth mentioning that, in general, atmoic LCIF $0$-monoids are not left-cancellative monoids, LCH $0$-monoids are neither left-cancellative monoids nor homogeneous monoids, and free $0$-monoids are not free monoids. 
    \end{remark}
    
    The mutually inverse mappings $\mathcal{M}_0$ and $\mathcal{\tilde{S}}$ between finitary colored semi-upho posets and atomic LCIF $0$-monoids are defined similarly to $\mathcal{M}$ and $\mathcal{\tilde{P}}$, respectively. The only difference is that the elements in the posets now correspond to \textbf{nonzero} elements in the monoids.
    \begin{cor}\label{cor:corrsp for semi-upho}
         The mutually inverse mappings $\mathcal{M}_0$ and $\mathcal{\tilde{S}}$ give a bijection between finitary colored semi-upho posets and atomic LCIF $0$-monoids. Moreover, $\N$-graded colored semi-upho posets correspond to LCH $0$-monoids, finite type $\N$-graded colored semi-upho posets correspond to finitely generated LCH $0$-monoids, and tree-like colored semi-upho posets correspond to free $0$-monoids.
    \end{cor}

The proof is similar to \Cref{thm:corrsp for finitary upho}, and the details are left to the readers.

\subsection{The Forgetful Mapping and Regularity of Upho Posets}\label{subsec:forgetful and regularity}

Through the \emph{forgetful mapping}, we establish an association between monoids and finitary colored upho posets with finitary upho posets.

The \emph{forgetful mapping} $\mathfrak{F}$ maps a finitary colored semi-upho poset $\tilde{S}=(S, \col_S)$ to $S$. This mapping is well-defined on colored upho posets, as colored upho posets are colored semi-upho posets. It can be conceptualized as forgetting the colors in the colored semi-upho posets.

\begin{remark}
    \Cref{fig:colorex} demonstrate that $\mathfrak{F}$ is not injective. Moreover, readers could verify that $\langle x_1,x_2\mid x_1^2=x_2^2 \rangle$ is the LCH monoid corresponding to the left colored upho poset, while $\langle x_1,x_2\mid x_1 x_2=x_2 x_1 \rangle$ corresponds to the right. These monoids are not isomorphic.

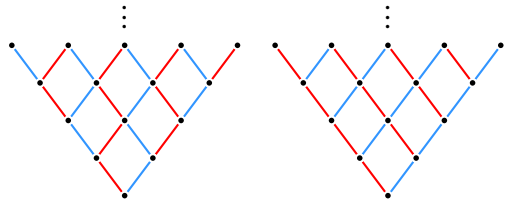
\begin{figure}[H]
\centering
    \input{Figure/colorex1}
    \caption{The forgetful mapping $\mathfrak{F}$ is not injective.}
    \label{fig:colorex}
\end{figure}
\end{remark}

Another point of inquiry is the surjectivity of the forgetful mapping, which remains an open question. In particular, the surjectivity of the restriction of the forgetful mapping from colored upho posets to upho posets also remains open.

Now we define \emph{regular semi-upho posets} to be the semi-upho posets in $\mathbf{im}\mathfrak{F}$, and \emph{regular upho posets} to be the upho posets in $\mathbf{im}\mathfrak{F}$. Moreover, an upho function is called \emph{regular upho} if it is the rank-generating function of a regular upho poset. The surjectivity of the forgetful mapping can be restated as the following conjectures.

\begin{conj}\label{conj:upho regular}
    All finitary upho posets are regular.
\end{conj}

\begin{conj}\label{conj:semiupho regular}
    All finitary semi-upho posets are regular.
\end{conj}

\subsection{Fundamental Applications}
The correspondence in \Cref{thm:corrsp for finitary upho} establishes a powerful connection between the algebraic and the combinatoric perspectives of upho posets. Here, we provide some fundamental applications to give the readers a quick taste.

On one hand, \Cref{thm:corrsp for finitary upho} provides combinatorial insight into problems such as the enumeration of elements in left-cancellative monoids. The following lemma is one such example.

\begin{lemma}
    Given an LCH monoid $M$, recall that we use $\mathbf{W}^M_{k}$ to denote the elements of length $k$ in $M$. Suppose $|\mathbf{W}^M_{k+1}|=|\mathbf{W}^M_{k}|$ for some $k\in \N$, then we have $|\mathbf{W}^M_n|=|\mathbf{W}^M_k|$ for all $n\ge k$.
\end{lemma}

\begin{proof}
By \Cref{cor:corrsp for N-graded upho}, it suffices to show that for a regular $\N$-graded upho poset $P$, if $|P_{k+1}| = |P_k|$, then $|P_n| = |P_k|$ for all $n \ge k$. In fact, this can be proved easily for all $\N$-graded upho posets. It suffices to show that if $|P_{k+1}| = |P_k| = r$, then $|P_{k+2}| = r$. Take an atom $s \in \mathcal{A}_P$. The isomorphism $V_s \xrightarrow{\sim} P$ implies that the number of elements in both the $k$-th and $(k+1)$-th layers above $s$ in $V_s$ is $r$. Moreover, since $|P_{k+1}| = r$, the elements in the $k$-th layer above $s$ in $V_s$ are precisely all these $r$ elements in $P_{k+1}$. Consequently, as the $(k+1)$-th layer above $s$ in $V_s$ also has $r$ elements, there are exactly $r$ elements in $P$ covering the $r$ elements in $P_{k+1}$, which implies $|P_{k+2}| = r$, as desired.
\end{proof}

On the other hand, \Cref{thm:corrsp for finitary upho} enables concrete calculations on upho posets using monoids from an algebraic perspective. Specifically, we can construct a variety of well-defined upho posets by showing the corresponding monoid is left-cancellative, as illustrated in \Cref{eg:not N graded} and \Cref{eg:not finite type}. Furthermore, our proof of \Cref{thm:main} in \Cref{sec:6} is entirely based on this method. 

Here, we introduce \emph{head-changing monoids}, the tool we use to prove \Cref{prop:type II}, an exceptional case of \Cref{thm:main}. Notably, we conjecture that all regular upho functions can be derived from such monoids, as discussed in detail in \Cref{subsec: char of upho function}.

\begin{definition}\label{def:head-changing}
    Let $M=\langle \mathcal{I}_{M}\mid R_M\rangle$ be an atomic monoid. A \emph{head-changing relation} is a defining relation $xW = yW$, where $x,y\in \mathcal{I}_{M}$ and $W\in M$. If all defining relations in $R_M$ are \emph{head-changing relations}, then such $M$ is called a head-changing monoid.
\end{definition}

\begin{lemma}\label{lemma:head-changing upho}
    Any head-changing monoid $M=\langle \mathcal{I}_{M}\mid R_M\rangle$ is an LCH monoid, and the corresponding poset $\mathcal{\tilde{P}}(M)$ is an $\N$-graded colored upho poset.
\end{lemma}

\begin{proof}
Since any defining relation is homogeneous, $M$ is homogeneous. Therefore, it suffices to show that $M$ is left-cancellative. Given any $W_1, W_2 \in \mathbf{F}(\mathcal{I}_{M})$ and $x \in \mathcal{I}_{M}$, we need to show that if $xW_1 = xW_2$ in $M$, then $W_1 = W_2$ in $M$.

By \Cref{prop:monoidtrans}, there exists a finite sequence  
\[
xW_1 \equiv y_1 V_1 \rightarrow y_2 V_2 \rightarrow \cdots \rightarrow y_n V_n \equiv xW_2
\]
where $y_1, y_2, \dots, y_n \in \mathcal{I}_{M}$ and $V_1, V_2, \dots, V_n \in \mathbf{F}(\mathcal{I}_{M})$, with each transition using a relation in $R_M$. We claim that each $V_i \rightarrow V_{i+1}$ is either an identity transition or uses a relation in $R_M$.

Suppose the transition $y_i V_i \rightarrow y_{i+1} V_{i+1}$ uses a relation that switches $y_i$. Then $V_i \equiv V_{i+1}$, making $V_i \rightarrow V_{i+1}$ an identity transition. Otherwise, if the transition $y_i V_i \rightarrow y_{i+1} V_{i+1}$ keeps $y_i$ fixed, then, as it uses a head-changing relation in $R_M$, $y_i$ is not involved in this relation. Thus, $y_i V_i \rightarrow y_{i+1} V_{i+1}$ induces a transition $V_i \rightarrow V_{i+1}$ using the same relation.

Therefore, the sequence 
\[
xW_1 \equiv y_1 V_1 \rightarrow y_2 V_2 \rightarrow \cdots \rightarrow y_n V_n \equiv xW_2
\]
induces another sequence 
\[
W_1 \equiv V_1 \rightarrow V_2 \rightarrow \cdots \rightarrow V_n \equiv W_2
\]
that only uses relations in $R_M$. By \Cref{prop:monoidtrans} again, we conclude that $W_1 = W_2$ in $M$.
\end{proof}

\section{Semi-Upho Posets and Their Rank-Generating Functions}\label{sec:4}

In this section, we develop a recursive framework using monoids to characterize all regular semi-upho functions and use it to prove \Cref{thm: log-concave rgf of semi-upho}.

\subsection{Regular Semi-Upho Posets and Tree-Like Semi-Upho Posets}\label{subsec: semi-upho spanning tree}

Tree-like semi-upho posets, as a special case of semi-upho posets, have a much simpler structure than arbitrary semi-upho posets. In this subsection, we prove that regular semi-upho functions are precisely the rank-generating functions of tree-like semi-upho posets. This simplification allows for a more accessible study of regular semi-upho functions.

\begin{prop}\label{prop:reg semiupho= tree gen}
    A formal power series $g(x)\in 1+x\Z_{\ge 0}[[x]$ is a regular semi-upho function if and only if it is the rank-generating function of a tree-like semi-upho poset.
\end{prop}

We split the proof of \Cref{prop:reg semiupho= tree gen} into the following two lemmas.

\begin{lemma}\label{lemma: trees are regular}
    Any tree-like semi-upho poset $S$ is regular.
\end{lemma}

\begin{proof}
    Fix an isoembedding $\psi_s: V_{s} \hookrightarrow S$ for each atom $s \in \mathcal{A}_{S}$, and denote the set of all such isoembeddings by $\Psi$. For any edge $(u,v) \in \mathcal{E}_{S}$, since $S$ is finitary, there exists, by induction, a sequence of isoembeddings in $\Psi$, denoted $\psi_1, \psi_2, \dots, \psi_n$, such that $\psi_n \circ \psi_{n-1} \circ \dots \circ \psi_1 (u,v) = (\hat{0}, t_{u,v})$ for some atom $t_{u,v} \in \mathcal{A}_{S}$. Since $S$ is tree-like, this sequence is unique. Define $\col_S(u,v) = t_{u,v}$ for each edge $(u,v) \in \mathcal{E}_{S}$, it can be easily verified that $\tilde{S} = (S, \col_S)$ is a colored semi-upho poset. Thus, $S=\mathfrak{F}({\tilde{S}})$ is regular, as desired.
\end{proof}

\begin{lemma}\label{lemma: regular semi-upho by tree}
    Given a finite type finitary colored semi-upho poset $\tilde{S} = (S, \col_S)$, there exists a colored tree-like semi-upho poset $\tilde{T} = (T, \col_T)$ such that the Hasse diagram of $T$ is a spanning tree of the Hasse diagram of $S$. Moreover, if $S$ is $\mathbb{N}$-graded, then the rank-generating functions satisfy $F_{S} = F_{T}$.
\end{lemma}

\begin{proof}
   By \Cref{cor:corrsp for semi-upho}, $\mathcal{M}_0(\tilde{S})$ is an atomic LCIF $0$-monoid, denoted $\langle \mathcal{I} \cup \{0\} \mid R \rangle$. Define an arbitrary total order on the finite set $I$, and extend it to a total order $<_u$ on $\mathbf{F}(I)$ such that $W_1 <_u W_2$ if either $\ell(W_1) < \ell(W_2)$ or $\ell(W_1) = \ell(W_2) = k$ and $W_1 <_k W_2$ under the $k$-lexicographic order on $\mathbf{F}_k(I)$. Define $M \coloneqq \langle  \mathcal{I} \cup \{0\} \mid R' \rangle$, where $R'$ includes relations $W_2 = 0$ for each $W_1 = W_2$ in $R_{\mathcal{M}_0(\tilde{S})}$ with $W_1 <_u W_2$.
   Since every nonzero element in $\mathcal{M}_0(\tilde{S})$ has a unique minimal word representation under $<_u$ corresponding to nonzero elements in $M$, the Hasse diagram of $\tilde{\mathcal{S}}(M)$ is a spanning tree of the Hasse diagram of $S$. Moreover, since $M$ only has $0$-defining relations, by \Cref{cor:corrsp for semi-upho}, $\tilde{\mathcal{S}}(M)$ is a colored tree-like semi-upho poset. Let $\tilde{T} \coloneqq \tilde{\mathcal{S}}(M)$, the proof is complete.
\end{proof}

\begin{remark}
    The ``finite type" condition in \Cref{lemma: regular semi-upho by tree} can be generalized to ``$ \mathcal{I}$ is well-ordered."
\end{remark}

\begin{example}
\Cref{fig:semiupho and tree} illustrates an example of \Cref{lemma: regular semi-upho by tree}. The $0$-monoid $\mathcal{M}_0(\tilde{S})$ corresponding to the colored semi-upho poset $\tilde{S}$ on the left is $\langle x_1, x_2, 0 \mid x_1 x_2 = x_2 x_1, x_1 0 = 0 x_1 = 0, x_2 0 = 0 x_2 = 0 \rangle$, and the monoid $M$ is given by $\langle x_1, x_2, 0 \mid x_2 x_1 = 0, x_1 0 = 0 x_1 = 0, x_2 0 = 0 x_2 = 0 \rangle$. The corresponding colored tree-like semi-upho poset on the right is $T \coloneqq \mathcal{S}(M)$, which shares the same rank-generating function as $\tilde{S}$.
    \begin{figure}[H]
\centering
    \input{Figure/semiupho_and_tree}
    \caption{A regular semi-upho poset and the corresponding tree-like semi-upho poset.}
    \label{fig:semiupho and tree}
\end{figure}
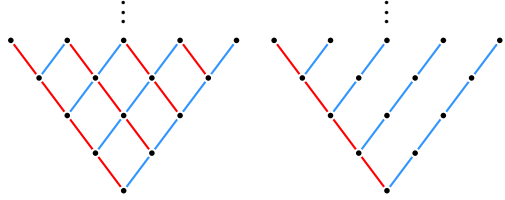
\end{example}

\begin{proof}[Proof of \Cref{prop:reg semiupho= tree gen}]
    This follows directly from \Cref{lemma: trees are regular} and \Cref{lemma: regular semi-upho by tree}.
\end{proof}

\subsection{Greedy $0$-Monoid Series and Tree-Like Semi-Upho Posets}\label{subsec: greedy 0 monoid}
In this subsection, we introduce \emph{greedy $0$-monoid series}, a recursive framework for proving that a given formal power series $g(x)$ is a regular semi-upho function.

\begin{definition}
Given a formal power series $g(x) = 1 + b_1 x + b_2 x^2 + \cdots\in 1 + x \mathbb{Z}_{\ge 0}[[x]]$, the \emph{greedy $0$-monoid series} $\{M_k^{0,g}\}_{k\ge 1}$ of $g(x)$ is a sequence of monoids defined recursively as follows:
\begin{itemize}
    \item Define $M_1^{0,g} \coloneqq \langle \mathbf{X} \cup \{0\} \mid 0 x_i = x_i 0 = 0 \text{ for any } x_i \rangle$, where the ordered set $\mathbf{X} \coloneqq \{x_1< x_2< \dots< x_{b_1}\}$.
    \item For $M_{k}^{0,g} = \langle \mathbf{X} \cup \{0\} \mid R_{k} \rangle$, if the inequality $|\mathbf{W}_{k+1}^{M_{k}^{0,g}}| \ge b_{k+1}$ holds, then define $M_{k+1}^{0,g} \coloneqq \langle \mathbf{X} \cup \{0\} \mid R_{k+1} \rangle$, where $R_{k+1}$ is obtained by extending $R_{k}$ with additional $0$-defining relations $W_i = 0$ for the largest $|\mathbf{W}_{k+1}^{M_{k}^{0,g}}| - b_{k+1}$ elements $W_i$ in $\mathbf{W}_{k+1}^{M_{k}^{0,g}}$ under the $(k+1)$-lexicographic order on $\mathbf{F}_{k+1}(\mathbf{X})$.
\end{itemize}
We call $M_k^{0,g}$ the \emph{$k$-th greedy $0$-monoid} of $g(x)$.
\end{definition}

We omit the superscript $g$ when the power series is clear from context. Notably, not all $g(x) \in 1 + x \mathbb{Z}_{\ge 0}[[x]]$ admit an infinite greedy $0$-monoid series,  and for any $k\ge 1$, $M_{k+1}^{0}$ is defined if and only if the $k$-th greedy $0$-monoid $M_{k}^{0}$ is defined and the inequality $|\mathbf{W}_{k+1}^{M_{k}^{0}}| \ge b_{k+1}$ holds.

It can be easily verified that, given $g(x) = 1 + b_1 x + b_2 x^2 + \cdots \in 1 + x \mathbb{Z}_{\ge 0}[[x]]$, the rank-generating function of the semi-upho poset $\mathcal{\tilde{S}}(M_k^0)$ corresponding to $M_k^0$ (see \Cref{cor:corrsp for semi-upho}) has the initial $k+1$ coefficients $1, b_1, b_2, \dots, b_k$. Thus, if $g(x)$ admits an infinite greedy $0$-monoid series, we define $M^0 \coloneqq \langle \mathbf{X} \cup \{0\} \mid \bigcup_{i \ge 1} R_i \rangle$, where $R_i$ is the set of defining relations of $M_i^0$. Consequently, the rank-generating function of $\mathcal{\tilde{S}}(M^0)$ is exactly $g(x)$, and hence by \Cref{cor:corrsp for semi-upho}, $g(x)$ is the rank-generating function of a tree-like semi-upho poset, which is also a regular semi-upho function by \Cref{prop:reg semiupho= tree gen}. This result can be formulated as the following proposition.

\begin{prop}\label{prop:greedy 0-series implies semi-upho function}
    Given a formal power series $g(x) \in 1 + x \mathbb{Z}_{\ge 0}[[x]]$, if $g(x)$ admits an infinite greedy $0$-monoid series $\{M_k^{0}\}_{k\ge 1}$, then $g(x)$ is a regular semi-upho function.
\end{prop}

An equivalent yet more practical version of \Cref{prop:greedy 0-series implies semi-upho function} is as follows.
\begin{prop}\label{prop:semi-upho upper bound}
    Given a formal power series $g(x) = 1 + b_1 x + b_2 x^2 + \cdots \in 1 + x \mathbb{Z}_{\ge 0}[[x]]$, if for any integer $k \ge 1$, the $k$-th greedy $0$-monoid $M_k^0$ is defined implies the inequality  $|\mathbf{W}_{k+1}^{M_{k}^{0}}| \ge b_{k+1}$ holds, then $g(x)$ is a regular semi-upho function.
\end{prop}

Moreover, we conjecture that \Cref{prop:greedy 0-series implies semi-upho function} and \Cref{prop:semi-upho upper bound} characterize all regular semi-upho functions.

\begin{conj}\label{conj:tree like=greedy series}
        A formal power series $g(x)\in 1 + x \mathbb{Z}_{\ge 0}[[x]]$ is a regular semi-upho function if and only if $g(x)$ admits an infinite greedy $0$-monoid series $\{M_k^{0}\}_{k\ge 1}$.
\end{conj}

\begin{conj}\label{conj:tree like recursive}
    A formal power series $g(x) = 1 + b_1 x + b_2 x^2 + \cdots \in 1 + x \mathbb{Z}_{\ge 0}[[x]]$ is a regular semi-upho function if and only if for any integer $k \ge 2$, the $k$-th greedy $0$-monoid $M_{k}^0$ is defined implies the inequality  $|\mathbf{W}_{k+1}^{M_{k}^{0}}| \ge b_{k+1}$ holds.
\end{conj}

\subsection{Counting Elements in Greedy $0$-Monoids}\label{subsec: counting greedy 0 monoid}
As shown in \Cref{prop:semi-upho upper bound}, to prove that a formal power series is a regular semi-upho function, it suffices to count nonzero elements in greedy $0$-monoids and compare them with the coefficients of the formal power series. In this subsection, we develop methods for counting nonzero elements in greedy $0$-monoids. Throughout this subsection, in each lemma and proposition, we consider a formal power series $g(x) = 1 + b_1 x + b_2 x^2 + \cdots \in 1 + x \mathbb{Z}_{\ge 0}[[x]]$ and its $k$-th greedy $0$-monoid $M_k^0=\langle\mathbf{X}\mid R_k\rangle$ , and denote $\mathbf{W}_n^{M_{k}^{0}}$ simply by $\mathbf{W}_n$.

\begin{lemma}\label{lemma: W=0 iff UW=0}
    Let $L \equiv x_{a_1} x_{a_2} \dots x_{a_n}$ be the maximal word under the $n$-lexicographic order of $\mathbf{F}_n(\mathbf{X})$ such that $L$ is nonzero in $M_k^0$. Then for any word $U \equiv x_{a_1} x_{a_2} \dots x_{a_{s-1}} y$ where $s \le n$ and $y <_X x_{a_s}$, we have $W = 0$ in $M_k^0$ if and only if $UW = 0$ in $M_k^0$ for any word $W \in \mathbf{F}_{n-s}(\mathbf{X})$.
\end{lemma}

\begin{proof}
    If $W = 0$, then clearly $UW = 0$. Conversely, suppose $UW = 0$. By \Cref{prop:monoidtrans}, there exists a $0$-defining relation $V = 0$ such that $V$ is a consecutive subword of $UW$. By the definition of greedy $0$-monoids, the fact that $V = 0$ is a $0$-defining relation implies that any word $V' \ge_{\ell(V)} V$ under the $\ell(V)$-lexicographic order of $\mathbf{F}_{\ell(V)}(\mathbf{X})$ is also $0$ in $M_k^0$. If the leftmost letter of $V$ is $y$, then $x_{a_s} x_{a_{s+1}} \dots x_{a_{s+\ell(V)-1}} >_{\ell(V)} V$, which implies $0 = x_{a_s} x_{a_{s+1}} \dots x_{a_{s+\ell(V)-1}} = L$, leading to a contradiction. If the leftmost letter of $V$ is $x_{a_i}$, then $x_{a_i} x_{a_{i+1}} \dots x_{a_{i+\ell(V)-1}} \ge_{\ell(V)} V$, resulting in $0 = x_{a_i} x_{a_{i+1}} \dots x_{a_{i+\ell(V)-1}} = L$, which is also a contradiction. Hence, the leftmost letter of $V$ is in $W$, which implies $W=0$, completing the proof.
\end{proof}

\begin{prop}\label{prop: split b_k}
    Let $L \equiv x_{a_1} x_{a_2} \dots x_{a_n}$ be the maximal word under the $n$-lexicographic order of $\mathbf{F}_n(\mathbf{X})$ such that $L$ is nonzero in $M_k^0$. Then for any integer $1\le s\le n$, we have 
    \[
    |\mathbf{W}_n|\le (a_1-1)|\mathbf{W}_{n-1}|+(a_2-1)|\mathbf{W}_{n-2}|+\dots +(a_{s-1}-1)|\mathbf{W}_{n+1-s}|+a_{s}|\mathbf{W}_{n-s}|,
    \] 
    and the equality holds when $s=n$.
    Specifically, when $n=k$, we have 
    \[
    b_k\le (a_1-1)b_{k-1}+(a_2-1)b_{k-2}+\dots +(a_{s-1}-1)b_{k+1-s}+a_{s}b_{k-s}
    \]
    for any $1\le s\le k$, and the equality holds when $s=k$.
\end{prop}

\begin{proof}
Any nonzero word of length $n$ starts with a unique prefix $x_{a_1} x_{a_2} \dots x_{a_{t-1}} y$ where $t \le s$ and $y <_X x_{a_{t}}$, or starts with $x_{a_1} x_{a_2} \dots x_{a_{s}}$. By \Cref{lemma: W=0 iff UW=0}, there are exactly $|\mathbf{W}_{n-t}|$ elements in $\mathbf{W}_{n}$ that start with each prefix $x_{a_1} x_{a_2} \dots x_{a_{t-1}} y$. Moreover, the number of elements in $\mathbf{W}_n$ starting with $x_{a_1} x_{a_2} \dots x_{a_{s}}$ is no more than the number of all elements in $\mathbf{W}_{n-s}$. As $b_{k-s} = |\mathbf{W}_{k-s}|$, we have
\[
    |\mathbf{W}_n|=(a_1-1)|\mathbf{W}_{n-1}|+(a_2-1)|\mathbf{W}_{n-2}|+\dots +(a_{s-1}-1)|\mathbf{W}_{n+1-s}|+(a_{s}-1)|\mathbf{W}_{n-s}|+|\mathbf{W}_{n-s}|,
\]
as desired. When $s=n$, the number of elements in $\mathbf{W}_n$ starting with $x_{a_1} x_{a_2} \dots x_{a_{s}}$ is exactly the number of all elements in $\mathbf{W}_{n-s}$, so the equality holds.
\end{proof}

\Cref{prop: split b_k} provides a recursive formula for computing $|\mathbf{W}_{n}|$ using the maximal nonzero element of length $n$. In practice, as shown in \Cref{prop:semi-upho upper bound}, the main interest is in calculating $|\mathbf{W}_{k+1}|$. While \Cref{prop: split b_k} gives an expression for $b_k$ based on the maximal nonzero element of length $k$, it only allows us to compute $|\mathbf{W}_{k+1}|$ using the maximal nonzero element of length $k+1$, which can differ significantly from that of length $k$. To establish a link between $|\mathbf{W}_{k+1}|$ and $b_k$, we provide the following technical yet practical proposition, offering an expression for $|\mathbf{W}_{k+1}|$ based on the maximal nonzero element of length $k$, rather than $k+1$.

\begin{prop}\label{prop: count k+1 from k}
Let $L \equiv x_{a_1} x_{a_2} \dots x_{a_k}$ be the maximal word in $\mathbf{F}_k(\mathbf{X})$ under the $k$-lexicographic order such that $L$ is nonzero in $M_k^0$. Then we have
    \[
    |\mathbf{W}_{k+1}| = (a_1 - 1)b_k + (a_2 - 1)b_{k-1} + \dots + (a_{s-1} - 1)b_{k+2-s} + a_{s}b_{k+1-s}
    \]
    for some integer $s$ with $1 \le s \le k$.
\end{prop}

To prove \Cref{prop: count k+1 from k}, we first prove the following two technical lemmas.

\begin{lemma}\label{lemma: exist subword W UW}
    Let $L \equiv x_{a_1} x_{a_2} \dots x_{a_k}$ be the maximal word in $\mathbf{F}_k(\mathbf{X})$ under the $k$-lexicographic order such that $L$ is nonzero in $M_k^0$. Then there exists a positive integer $s$ with $1 \le s \le k$ such that $U \equiv x_{a_1} x_{a_2} \dots x_{a_{s}}$ satisfies that $W = 0$ in $M_k^0$ if and only if $UW = 0$ in $M_k^0$ for any word $W \in \mathbf{F}_{k+1-s}(\mathbf{X})$.
\end{lemma}

\begin{proof}
    If $W = 0$, then clearly $UW = 0$. Suppose the converse direction is false. We prove by contradiction. Let $V_i$ denote the maximal word in $\mathbf{F}_{k+2-i}(\mathbf{X})$ under the $(k+2-i)$-lexicographic order such that $V_i \neq 0$ in $M_k^0$ for $2 \le i \le k$. 

    Claim 1: For any $2 \le i \le k$, we have $x_{a_1} x_{a_2} \dots x_{a_{i-1}} V_i = 0$. 

    By our assumption that the converse direction of the lemma is false, for any $2 \le i \le k$, there exists a word $W \in \mathbf{F}_{k+2-i}(\mathbf{X})$ such that $W \neq 0$ but $x_{a_1} x_{a_2} \dots x_{a_{i-1}} W= 0$ in $M_k^0$. Since $W \neq 0$, there must be a $0$-defining relation $x_{a_j} x_{a_{j+1}} \dots x_{a_{i-1}} W' = 0$ for some prefix $W'$ of $W$. By the construction of greedy $0$-monoids, this implies that all words greater than or equal to $x_{a_j} x_{a_{j+1}} \dots x_{a_{i-1}} W'$ under the $(i-j+\ell(W'))$-lexicographic order are also zero, and thus $x_{a_j} x_{a_{j+1}} \dots x_{a_{i-1}} V'_i=0$ for the prefix $V'_i$ of $V_i$ such that $\ell(V'_i)=\ell(W')$. Hence, we have $x_{a_1} x_{a_2} \dots x_{a_{i-1}} V_i = 0$. This completes the proof of Claim 1.

    Claim 2: For any $2 \le i \le k$, we have $V_i^- >_{k+1-i} x_{a_i} x_{a_{i+1}} \dots x_{a_k}$, where $V_i \equiv V_i^- v_i$ with $v_i$ being the last letter of $V_i$.

    We prove Claim 2 by induction on $i$. For $i = 2$, by Claim 1, we know that $x_{a_1} V_2 = 0$. Since $V_2 \neq 0$, there exists a $0$-defining relation $x_{a_1} V'_2=0$ such that $V'_2$ is a prefix of $V_2$. As no $0$-defining relation $W=0$ in $M_k^0$ satisfies $\ell(W)\ge k+1$, we have $\ell(V'_2)\le \ell(V_2)-1=\ell(V_2^-)$. Hence $x_{a_1} x_{a_2} \dots x_{a_{\ell(V'_2)+1}}$ is a prefix of $x_{a_1} x_{a_2} \dots x_{a_{k}}$, resulting in $x_{a_1} x_{a_2} \dots x_{a_{\ell(V'_2)+1}}\neq 0$ in $M_k^0$. Moreover, since $x_{a_1} V'_2=0$ is a $0$-defining relation, by the construction of greedy $0$-monoids, we have $x_{a_1} V'_2>_{\ell(V'_2)+1} x_{a_1} x_{a_2} \dots x_{a_{\ell(V'_2)+1}}$, and hence $V'_2>_{\ell(V'_2)}x_{a_2} \dots x_{a_{\ell(V'_2)+1}}$. Therefore, we have $V_2^{-} >_{k-1} x_{a_2} x_{a_{2}} \dots x_{a_{k}}$ and complete the proof of the base case.

     Suppose we have proved Claim 2 for all $i'$ such that $2\le i'<i$. By Claim 1, $x_{a_1} x_{a_2} \dots x_{a_{i-1}}V_i=0$. 
     
     Case 1: If there exists a $0$-defining relation $x_{a_{i-r}} x_{a_{i-r+1}} \dots x_{a_{i-1}}V'_i=0$ ($1\le r\le i-1$), where $V'_i$ is a prefix of $V'_i$ such that $\ell(V'_i)\le \ell(V_i)-1=\ell(V_i^-)$, then $x_{a_{i-r}} x_{a_{i-r+1}} \dots x_{a_{i-1+\ell(V'_i)}}$ is a consecutive subword of $x_{a_1} x_{a_2} \dots x_{a_{k}}$. Since $x_{a_1} x_{a_2} \dots x_{a_{k}}\neq 0$ in $M_k^0$, we have $x_{a_{i-r}} x_{a_{i-r+1}} \dots x_{a_{i-1+\ell(V'_i)}}\neq 0$ in $M_k^0$, and as $x_{a_{i-r}} x_{a_{i-r+1}} \dots x_{a_{i-1}}V'_i=0$ is a $0$-defining relation, by the construction of greedy $0$-monoids, we have $x_{a_{i-r}} x_{a_{i-r+1}} \dots x_{a_{i-1}}V'_i>_{r+\ell(V'_i)} x_{a_{i-r}} x_{a_{i-r+1}} \dots x_{a_{i-1+\ell(V'_i)}}$, and hence $V_i^->_{k+1-i} x_{a_i} x_{a_{i+1}} \dots x_{a_{k}}$.

     Case 2: Otherwise, since $V_i \neq 0$ and no $0$-defining relation $W=0$ in $M_k^0$ satisfies $\ell(W)\ge k+1$, there exists a $0$-defining relation $x_{a_{i-r}} x_{a_{i-r+1}} \dots x_{a_{i-1}}V_i=0$ ($1\le r\le i-2$). Since $V_{i-r}\neq 0$, by the construction of greedy $0$-monoids, we have $x_{a_{i-r}} x_{a_{i-r+1}} \dots x_{a_{i-1}}V_i>_{k+2-i+r}V_{i-r}$. Moreover, by induction hypothesis, $V_{i-r}\equiv V_{i-r}^- v_{i-r}>_{k+2-i+r} x_{a_{i-r}} x_{a_{i-r+1}} \dots x_{a_{k}}v_{i-r}$. Hence $V_i>_{k+2-i} x_{a_{i}} x_{a_{i+1}} \dots x_{a_{k}}v_{i-r}$, and thus $V_i^- \ge_{k+2-i} x_{a_{i}} x_{a_{i+1}} \dots x_{a_{k}}$. The equality is obtained only if $x_{a_{i-r}} x_{a_{i-r+1}} \dots x_{a_{i-1}}V_i^-\equiv V_{i-r}^- \equiv x_{a_{i-r}} x_{a_{i-r+1}} \dots x_{a_{k}}$, contradicting with the induction hypothesis $V_{i-r}^- >_{k+1-i+r} x_{a_{i-r}} x_{a_{i-r+1}} \dots x_{a_{k}}$. Hence, we have $V_i^- >_{k+2-i} x_{a_{i}} x_{a_{i+1}} \dots x_{a_{k}}$. Therefore, the proof of Claim 2 is complete.

    Claim 3: There does not exist a $0$-defining relation $x_{a_i} x_{a_{i+1}} \dots x_{a_k} y = 0$ for any $1 \le i \le k$ and $y \in X$.

    The case $i = 1$ is trivial, as no $0$-defining relation $W=0$ in $M_k^0$ satisfies $\ell(W)\ge k+1$. For $i \ge 2$, by Claim 2, $x_{a_i} x_{a_{i+1}} \dots x_{a_k} y <_{k+2-i} V_i$. Thus, if $x_{a_i} x_{a_{i+1}} \dots x_{a_k} y = 0$, then $V_i = 0$, contradicting the nonzero nature of $V_i$. This completes Claim 3.

    Hence, by Claim 3 and the fact that $x_{a_1} x_{a_2} \dots x_{a_k}\neq 0$, we have $x_{a_1} x_{a_2} \dots x_{a_k} y=0$ if and only if $y=0$, contradicting our assumption. This completes the proof of the lemma.
\end{proof}

\begin{remark}
     ``For any word $W \in \mathbf{F}_{k+1-s}(\mathbf{X})$" in \Cref{lemma: exist subword W UW} cannot be replaced by ``for any word $W$ of length no more than $k+1-s$."
\end{remark}

\begin{lemma}\label{lemma: k+1 W=0 iff VW=0}
     Let $L \equiv x_{a_1} x_{a_2} \dots x_{a_k}$ be the maximal word in $\mathbf{F}_k(\mathbf{X})$ under the $k$-lexicographic order such that $L$ is nonzero in $M_k^0$. Suppose $U \equiv x_{a_1} x_{a_2} \dots x_{a_{s}}$ is a prefix of $L$ such that $W = 0$ in $M_k^0$ if and only if $UW = 0$ in $M_k^0$ for any word $W \in \mathbf{F}_{k+1-s}(\mathbf{X})$, then for any $V\equiv x_{a_1} x_{a_2} \dots x_{a_{t-1}} y$ where $t\le s$ and $y <_X x_{a_t}$, we have $W = 0$ in $M_k^0$ if and only if $VW = 0$ in $M_k^0$ for any word $W \in \mathbf{F}_{k+1-t}(\mathbf{X})$.
\end{lemma}

\begin{proof}
If $W=0$, then clearly $VW=0$. Conversely, suppose otherwise that $W \neq 0$ and $VW = 0$. Then there exists a $0$-defining relation given by a consecutive subword of $VW$. The $0$-defining relation cannot be of the form $x_{a_i} x_{a_{i+1}} \dots x_{a_{j}}=0$ since $L\neq 0$. Moreover, since $W\neq 0$, this $0$-defining relation must be of the form $V'yW' = 0$, where $V'y$ is the last few letters of $V$, and $W'$ is a prefix of $W$ (both $V'$ and $W'$ can be the empty word). 

If $\ell(W')\neq \ell(W)$, then $V'yW'<_{\ell(V')+\ell(W')+1} V'x_{a_t} x_{a_{t+1}}\dots x_{a_{t+\ell(W')}}$. By the construction of greedy $0$-monoid, $V'x_{a_t} x_{a_{t+1}}\dots x_{a_{t+\ell(W')}}=0$. Moreover, since it is a consecutive subword of $L$, we get $L=0$, leading to a contradiction.

If $\ell(W')= \ell(W)$, then this $0$-defining relation is exactly $V'yW=0$. Note that $V'yW<_{\ell(V')+\ell(W')+1}V' x_{a_t}x_{a_{t+1}}\dots x_{a_{s}}W''$, where $W''$ is the last $\ell(W)-s+t=k+1-s$ letters of $W$. By the construction of greedy $0$-monoid, we have $V' x_{a_t}x_{a_{t+1}}\dots x_{a_{s}}W''=0$. Since $V' x_{a_t}x_{a_{t+1}}\dots x_{a_{s}}$ is the last few letters of $U$, we have $UW''=0$, and hence $W''=0$ as $W'' \in  \mathbf{F}_{k+1-s}(\mathbf{X})$. But this implies $W=0$, which is also a contradiction.
\end{proof}

Now we are ready to prove \Cref{prop: count k+1 from k}.

\begin{proof}[Proof of \Cref{prop: count k+1 from k}]
Any nonzero word of length $k+1$ starts with a unique prefix $x_{a_1} x_{a_2} \dots x_{a_{t-1}} y$ where $t \le s$ and $y <_X x_{a_{t}}$, or starts with $x_{a_1} x_{a_2} \dots x_{a_{s}}$. By \Cref{lemma: k+1 W=0 iff VW=0}, there are exactly $|\mathbf{W}_{k+1-t}| = b_{k+1-t}$ elements in $\mathbf{W}_{k+1}$ that start with each prefix $x_{a_1} x_{a_2} \dots x_{a_{t-1}} y$, and moreover, by \Cref{lemma: exist subword W UW}, there are exactly $|\mathbf{W}_{k+1-s}|=b_{k+1-s}$ elements that start with the prefix $x_{a_1} x_{a_2} \dots x_{a_{s}}$. Hence, we have
\[
    |\mathbf{W}_{k+1}| = (a_1 - 1)b_k + (a_2 - 1)b_{k-1} + \dots + (a_{s-1} - 1)b_{k+2-s} + a_{s}b_{k+1-s}
\]
for some $1 \le s \le k$, as desired.
\end{proof}

\subsection{Log-Concavity Implies Regular Semi-Upho}
With the results in \Cref{subsec: greedy 0 monoid} and \Cref{subsec: counting greedy 0 monoid}, we are now ready to prove the main theorem of \Cref{sec:4}, which states that log-concave power series are all regular semi-upho functions.

\begin{proof}[Proof of \Cref{thm: log-concave rgf of semi-upho}]
According to \Cref{prop:semi-upho upper bound}, it suffices to show $|\mathbf{W}_{k+1}^{M_{k}^{0}}| \ge b_{k+1}$ for any $k\in \Z$.  Moreover, it suffices to prove the case where $b_{k+1}\neq 0$, and since the log-concave formal power series does not have internal zeros, we have $b_1, b_2, \dots b_{k}\neq 0$ and\[\frac{b_1}{1}\ge \frac{b_2}{b_1}\ge \dots \frac{b_{k}}{b_{k-1}}\ge \frac{b_{k+1}}{b_{k}}.\]
By \Cref{prop: count k+1 from k}, we have\[
     |\mathbf{W}_{k+1}^{M_{k}^{0}}| = (a_1 - 1)b_k + (a_2 - 1)b_{k-1} + \dots + (a_{s-1} - 1)b_{k+2-s} + a_{s}b_{k+1-s}
    \]
    for some integer $s$ with $1 \le s \le k$.
    Moreover, by \Cref{prop: split b_k}, we have
   \[
    b_k\le (a_1-1)b_{k-1}+(a_2-1)b_{k-2}+\dots +(a_{s-1}-1)b_{k+1-s}+a_{s}b_{k-s}.
    \]
    Therefore, we have
    \begin{align*}
        |\mathbf{W}_{k+1}^{M_{k}^{0}}|=&(a_1 - 1)b_k + (a_2 - 1)b_{k-1} + \dots + (a_{s-1} - 1)b_{k+2-s} + a_{s}b_{k+1-s}\\
        =&(a_1-1)b_{k-1}\frac{b_k}{b_{k-1}}+(a_2-1)b_{k-2}\frac{b_{k-1}}{b_{k-2}}+\cdots+a_sb_{k-s}\frac{b_{k+1-s}}{b_{k-s}}\\
        \geq &((a_1-1)b_{k-1}+(a_2-1)b_{k-2}+\dots +(a_{s-1}-1)b_{k+1-s}+a_{s}b_{k-s})\frac{b_{k}}{b_{k-1}}\\
        \geq &b_k \cdot \frac{b_k}{b_{k-1}}
        \geq b_{k}\cdot \frac{b_{k+1}}{b_{k}}
        = b_{k+1},
    \end{align*}  
    as desired.
\end{proof}

\begin{remark}
    As shown in \Cref{example:feiposet}, the rank-generating function of Fei's poset is not log-concave, indicating that not all regular semi-upho posets or regular upho functions are log-concave. 
\end{remark}

However, \Cref{thm: log-concave rgf of semi-upho} remains a highly nontrivial result. In \Cref{subsec: char of upho function}, we establish a framework for regular upho functions using \emph{greedy LCH monoid series}, analogous to the framework for regular semi-upho functions in \Cref{subsec: greedy 0 monoid}. Despite the apparent similarities, and although almost all weakly increasing log-concave formal power series we observed in practice are regular upho functions, they do not universally satisfy this property.

\begin{remark}
    Using SageMath \cite{sagemath}, we found that $1,4,11,30$ cannot be the first few coefficients of any upho function, although $f(x) = 1 + 4x + 11x^2 + 30x^3 + 30x^4 + 30x^5 + \dots$ is a weakly increasing log-concave formal power series.
\end{remark}

The complexity of regular upho functions compared to regular semi-upho functions stems from the greater intricacy of LCH monoids (even for head-changing LCH monoids) relative to free $0$-monoids. Even with a full resolution of the conjectures in \Cref{subsec: char of upho function}, directly describing the coefficients of upho functions remains difficult due to the challenges in counting elements in LCH monoids. Consequently, in the following sections, we present an alternative approach to constructing new regular upho functions from known ones.

\section{Convolution of Posets}\label{sec:5}

In this section, we define the \emph{$\bx$-convolution} of a colored upho poset $P$ with a colored tree-like semi-upho poset $S$, denoted $P \rtimes_{\bx} S$. Our main theorem states that $P \rtimes_{\bx} S$ is always a colored upho poset. Moreover, if the rank-generating functions are well-defined, then we have $F_{P \rtimes_{\bx} S} = F_{P} F_{S}$. This result can be reformulated as \Cref{thm:upho times semi-upho}, the product of any regular upho function and any regular semi-upho function remains a regular upho function. This theorem provides a more flexible method, compared to \Cref{lemma:produpho}, for generating new upho functions from existing ones.

\subsection{Convolution of Monoids}\label{subsec: conv of monoids}

In this subsection, we define and study the monoid version of the \emph{$\bx$-convolution}.

\begin{definition}\label{def:convolution of monoids}
   Let $M_1=\langle \mathcal{I}_{M_1}\mid R_{M_1}\rangle$ be an atomic LCIF monoid, $M_2=\langle \mathcal{I}_{M_2}\cup \{0\}\mid R_{M_2}\rangle$ be a free $0$-monoid, and $\bx:\mathcal{I}_{M_2}\rightarrow M_1\setminus\{e\}$ be a mapping. The \emph{$\bx$-convolution} of $M_1$ with $M_2$, denoted $M_1\rtimes_{\bx} M_2$, is defined by $\langle \mathcal{I}_{M_1}\cup \mathcal{I}_{M_2}\mid R_{M_1\rtimes_{\bx} M_2}\rangle$, where $R_{M_1\rtimes_{\bx} M_2}$ consists of the following relations:
    \begin{enumerate}
    \item All defining relations in $R_{M_1}$; \label{gen:case1}
    \item Relations $y_i Y_j =\bx(y_i) Y_j$, where $y_i \in \mathcal{I}_{M_2}, Y_j \in \mathbf{F}(\mathcal{I}_{M_2})$, and the relation $y_i Y_j =0$ is in $R_{M_2}$;  \label{gen:case2}
    \item Relations $y_i x_j =\bx(y_i) x_j$, where $x_j \in \mathcal{I}_{M_1}$ and $y_i\in \mathcal{I}_{M_2}$. \label{gen:case3}
    \end{enumerate}
\end{definition}

It can be easily verified that the $\bx$-convolution of $M_1$ with $M_2$ does not depend on the choice of $R_{M_1}$ and $R_{M_2}$ in their presentations.

\begin{example}\label{eg:eg for fig upho times semiupho}
    Let $M_1=\langle x\mid \emptyset \rangle$, $M_2=\langle 0, y_1, y_2\mid R_{M_2}\rangle,$ where
    \[R_{M_2}=\{y_1^3=y_1^2y_2=y_1y_2y_1=y_1y_2^2=y_2y_1^2=y_2y_1y_2=y_2^2=0, 0y_1=y_10=0, 0y_2=y_20=0\},\] and $\bx(y_1)=x$, $\bx(y_2)=x$.
    Then $M_1\rtimes_{\bx} M_2=\langle x, y_1, y_2\mid  R\rangle$, where \[ R=\{y_1^3=y_2y_1^2=xy_1^2, y_1^2y_2=y_2y_1y_2=xy_1y_2, y_1y_2y_1=xy_2y_1, y_1y_2^2=xy_2^2, y_2^2=xy_2, y_1x=y_2x=x^2\}.\]
\end{example}

Our main theorem in this subsection is the following result.

\begin{theorem}\label{thm: LCIF monoid convolution}
    Let $M_1$ be an atomic LCIF monoid, and $M_2$ be a free $0$-monoid. Then $M_1\rtimes_{\bx} M_2$ is an atomic LCIF monoid for any $\bx:\mathcal{I}_{M_2}\rightarrow M_1\setminus\{e\}$. Moreover, if $M_1$ is an LCH monoid, and $\im \bx \subseteq \mathcal{I}_{M_1}$, then $M_1\rtimes_{\bx} M_2$ is an LCH monoid.
\end{theorem}

In the remainder of this subsection, we prove \Cref{thm: LCIF monoid convolution}. Unless otherwise specified, $M_1$ refers to an atomic LCIF monoid, $M_2$ refers to a free $0$-monoid, $M$ refers to $M_1\rtimes_{\bx} M_2$, and $R$ refers to $R_{M_1\rtimes_{\bx} M_2}$. For simplicity, we use the following notation to denote defining relations in $R$.

\begin{notation}
    Defining relations \ref{gen:case1}, \ref{gen:case2}, and \ref{gen:case3} in \Cref{def:convolution of monoids} are denoted by \emph{class I relations}, \emph{class II relations}, and \emph{class III relations}, respectively.
\end{notation}

We introduce the following definitions to standardize the word representations of elements in $M$.

\begin{definition}
     Given an element $W\in M$, a \emph{separate word} of $W$ is a word of the form $XY$ such that $W=XY$, $X \in \mathbf{F}(\mathcal{I}_{M_1})$, and $Y \in \mathbf{F}(\mathcal{I}_{M_2})$. Moreover, if the length of $Y$ is the shortest among all separate word of $W$, then $XY$ is called a \emph{standard word} of $W$, and the length of $Y$ is called the \emph{depth} of $W$. We use $d(W)$ to denote the depth of $W$.
\end{definition}

Before proving \Cref{thm: LCIF monoid convolution}, we first introduce a series of technical lemmas to help with handling the left-cancellative property.

\begin{lemma}\label{lemma:existence of standard word}
    Given an element $W\in M$, there exists a standard word $XY$ of $W$. 
\end{lemma}

\begin{proof}
      Note that given any word representation $W'$ of $W$, class III relations allows us to replace any $y_j\in \mathcal{I}_{M_2}$ that appears to the left of the rightmost letter $x_i\in \mathcal{I}_{M_1}$ in $W'$ by $\bx(y_j)\in \mathcal{I}_{M_1}$. Therefore, there exists a separate word of $W$. Let $XY$ be a separate word of $W$ such that $\ell(Y)$ is the shortest among all separate words of $W$, then $XY$ is a standard word of $W$.
\end{proof}

\begin{lemma}\label{lemma:last d(W) letters}
    Given elements $W, V\in M$ and a standard word $XY$ of $W$, the last $d(W)$ letters of any word representation $U$ of $VW$ are exactly $Y$. Specifically, the last $d(W)$ letters of any word representation $W'$ of $W$ are exactly $Y$.
\end{lemma}

\begin{proof}
    Suppose otherwise, by \Cref{prop:monoidtrans}, there exists a finite sequence \[VXY\equiv W_1\rightarrow W_2\rightarrow \cdots \rightarrow W_n\equiv U\] such that each step of transition uses relations in $R$. Let $W_{i+1}$ be the first word from the left of this sequence such that the last $d(W)$ letters are not exactly $Y$ ($1\le i \le n-1$). Then the last $d(W)$ letters of $W_{i}$ are exactly $Y$. Since the transition $W_i\rightarrow W_{i+1}$ only uses relations in $R$, and it replaces some letter in $Y$ by other letters, it must use a class II relation. Suppose this relation replaces $y_i$ by $\bx(y_i)$, then this implies that in the standard word $XY$ of $W$, we can also replace this $y_i$ by $\bx(y_i)$. Then apply class III relations, we can obtain a separate word of $W$ which has fewer $y_j$ in the rightmost part, contradicting with the fact that $XY$ is a standard word of $W$.
\end{proof}

In \Cref{lemma:last d(W) letters}, by taking $V=x_i\in \mathcal{I}_{M_1}$, we directly obtain the following corollary.

 \begin{cor}\label{cor:standard word extension}
     Given an element $W\in M$ and $x_i\in \mathcal{I}_{M_1}$, suppose $XY$ is a standard word of $W$, then $x_iXY$ is a standard word of $x_iW$.
\end{cor}

\begin{lemma}\label{lemma:standard word the same tail}
        Given an element $W\in M$ and two standard words $X_1Y$ and $X_2Y$ of $W$, we have $X_1=X_2$ in $M_1$.
\end{lemma}
\begin{proof}
 By \Cref{prop:monoidtrans} and \Cref{lemma:last d(W) letters}, there exists a finite sequence  \[X_1Y\equiv W'_1Y\rightarrow W'_2Y\rightarrow \cdots \rightarrow W'_nY\equiv X_2Y\] such that each step of transition uses a relation in $R$. Now we define a new sequence \[X_1Y\equiv W_1Y\rightarrow W_2Y\rightarrow \cdots \rightarrow W_nY\equiv X_2Y\] where each $W_i$ is obtained by replacing all letters $y_j\in \mathcal{I}_{M_2}$ in $W'_i$ by $\bx(y_j)$. If $W'_iY\rightarrow W'_{i+1}Y$ uses a class I relation, then $W_iY\rightarrow W_{i+1}Y$ also uses a class I relation. If $W'_iY\rightarrow W'_{i+1}Y$ uses a class II or class III relation, then $W_iY\rightarrow W_{i+1}Y$ is an identity transition. Hence, the sequence \[X_1Y\equiv W_1Y\rightarrow W_2Y\rightarrow \cdots \rightarrow W_nY\equiv X_2Y\] is well-defined and only uses class I relations.

Since none of the transitions involve the letters $y_i\in \mathcal{I}_{M_2}$ in $Y$, this induces a finite sequence \[X_1\equiv W_1\rightarrow W_2\rightarrow \cdots \rightarrow W_n\equiv X_2,\] where each element is in $M_1$, and each step of transition only uses a class I relation, which is a relation in $R_{M_1}$. Therefore, we have $X_1=X_2$ in $M_1$.
\end{proof}

    \begin{lemma}\label{lemma:head x_i}
         Given elements $W_1, W_2$ in $M$ and $x_i\in \mathcal{I}_{M_1}$, suppose $x_iW_1=x_iW_2$ in $M$, then $W_1=W_2$.
    \end{lemma}
 \begin{proof}
     Let $X_1Y_1$ be a standard word of $W_1$, and $X_2Y_2$ be a standard word of $W_2$. By \Cref{cor:standard word extension}, $x_iX_1Y_1$ is a standard word of $x_iW_1$, and $x_iX_2Y_2$ is a standard word of $x_iW_2$. By \Cref{lemma:last d(W) letters}, we have $Y_1\equiv Y_2$, and we denote both by $Y$. By \Cref{lemma:standard word the same tail}, it follows that $x_iX_1=x_iX_2$ in $M_1$. Now since $M_1$ is left-cancellative, we have $X_1=X_2$ in $M_1$, and thus $X_1Y=X_2Y$ in $M$. Moreover, since $X_1Y$ and $X_2Y$ are standard words of $W_1$ and $W_2$ respectively, we have $W_1=X_1Y=X_2Y=W_2,$ as desired.
 \end{proof}

 \begin{lemma}\label{lemma:equiv Y}
     Given an elements $W$ in $M$, suppose $Y\in \mathbf{F}(\mathcal{I}_{M_2})$ is a standard word of $W$, then $W\equiv Y$.
 \end{lemma}
 \begin{proof}
     Since $Y$ is also a word representation of itself, by \Cref{lemma:last d(W) letters}, the last $\ell(Y)$ letters of any word representation of $Y$ are exactly $Y$. This implies that no class II relations can be applied to $Y$, and hence no relations in $R$ can be applied to $Y$. Therefore, $Y$ is the only word representation of $Y$, and $W\equiv Y$.
 \end{proof}
\begin{lemma}\label{lemma:head y_i}
        Given elements $W_1, W_2$ in $M$ and $y_i\in \mathcal{I}_{M_2}$, suppose $y_iW_1=y_iW_2$ in $M$, then $W_1=W_2$.
\end{lemma}
\begin{proof}
    If $y_iW_1=\bx(y_i)W_1$ and $y_iW_2=\bx(y_i)W_2$, by \Cref{lemma:head x_i}, the proof is complete. 
    
    Otherwise, without loss of generality, suppose $y_iW_1\neq \bx(y_i)W_1$. Let $X_1Y_1$ be a standard word of $W_1$. Assuming $X_1$ is not the empty word $e$, then $y_iW_1=y_iX_1Y_1=\bx(y_i)X_1Y_1=\bx(y_i)W_1$, which leads to a contradiction. Therefore, $X_1=e$, and by \Cref{lemma:equiv Y}, $W_1\equiv Y_1$. By \Cref{lemma:last d(W) letters}, the last $\ell(Y_1)$ letters of any word representation of $y_iW_1$ are $Y_1$. Now, assuming that $y_iY_1$ is not a standard word of $y_iY_1$, there exists a class II relation in $R$ that replaces $y_i$ by $\bx(y_i)$, which implies $y_iW_1\equiv y_iY_1=\bx(y_i)Y_1\equiv \bx(y_i)W_1$, contradicting our assumption. Hence, $y_iY_1$ must be a standard word of $y_iY_1$. Since $y_iW_1=y_iW_2$, $y_iY_1$ is also a standard word of $y_iW_2$, and by \Cref{lemma:equiv Y}, $y_iW_2\equiv y_iY_1$. Therefore, $W_2\equiv Y_1\equiv W_1$, as desired.
\end{proof}

Now we are ready to prove \Cref{thm: LCIF monoid convolution}.
   
\begin{proof}[Proof of \Cref{thm: LCIF monoid convolution}]
    First, we show that $M$ is invertible-free. Suppose $W_1W_2=e$, where $W_1,W_2\in M$. We want to show that $W_1=W_2=e$. Let $X_1Y_1$ and $X_2Y_2$ be standard words of $W_1$ and $W_2$, respectively. Then $X_1Y_1X_2Y_2=e$. By \Cref{lemma:last d(W) letters}, the last $\ell(Y_2)$ letters of $e$ is exactly $Y_2$, and hence $Y_2=e$.
    
    Assuming that $X_2\neq e$, then we can apply class III relations on the word $X_1Y_1X_2$ to replace each $y_i$ in $Y_1$ by $\bx(y_i)$. Denote this new word by $V$. By \Cref{prop:monoidtrans}, there exists a finite sequence  \[V\equiv V'_1\rightarrow V'_2\rightarrow \cdots \rightarrow V'_n\equiv e\] such that each step of transition uses a relation in $R$. Now we define a new sequence \[V\equiv V_1\rightarrow V_2\rightarrow \cdots \rightarrow V_n\equiv e\] where each $V_i$ is obtained by replacing all letters $y_j\in \mathcal{I}_{M_2}$ in $V'_i$ by $\bx(y_j)$. If $V'_i\rightarrow V'_{i+1}$ uses a class I relation, then $V_i\rightarrow V_{i+1}$ also uses a class I relation. If $V'_i\rightarrow V'_{i+1}$ uses a class II or class III relation, then $V_i\rightarrow V_{i+1}$ is an identity transition. Hence, the sequence \[V\equiv V_1\rightarrow V_2\rightarrow \cdots \rightarrow V_n\equiv e\] is well-defined and only uses class I relations, which are relations in $R_{M_1}$. Thus, $V=e$ in $M_1$. Since $M_1$ is invertible-free, we have $X_1=X_2=e$, contradicting our assumption that $X_2\neq e$.

    Hence $X_2=e$, and we get $W_2=e$, and $W_1=W_1W_2=e$. Therefore, $M$ is invertible-free.
    
    By \Cref{lemma:head x_i} and \Cref{lemma:head y_i}, $M$ is left-cancellative. Finally, if $M_1$ is an LCH monoid, and $\im \bx \subseteq \mathcal{I}_{M_1}$, then any relation in $M$ are homogeneous, and hence $M$ is an LCH monoid.
\end{proof}

Finally, if $M_1$ is an LCH monoid, and $\im \bx \subseteq \mathcal{I}_{M_1}$, we give the enumeration of distinct elements of a given length in $M$.

\begin{prop}\label{prop: enum of conv monoid}
     Let $M_1$ be an LCH monoid, and $M_2$ be a free $0$-monoid. If $\im \bx \subseteq \mathcal{I}_{M_1}$, then for any nonnegative integer $n$, we have
     \[|\mathbf{W}_n^M|=|\mathbf{W}_0^{M_1}||\mathbf{W}_n^{M_2}|+|\mathbf{W}_1^{M_1}||\mathbf{W}_{n-1}^{M_2}|+\dots+ |\mathbf{W}_{n-i}^{M_1}||\mathbf{W}_{i}^{M_2}| +\dots +|\mathbf{W}_n^{M_1}||\mathbf{W}_{0}^{M_2}|.\]
\end{prop}
\begin{proof}
By \Cref{lemma:last d(W) letters} and \Cref{lemma:standard word the same tail}, given an element $W \in \mathbf{W}_n^M$ of depth $d(W) = i$, any word representation of it is of the form $XY$, where $X \in \mathbf{F}(\mathcal{I}_{M_1})$, $Y \in \mathbf{F}(\mathcal{I}_{M_2})$, and $\ell(Y) = i$. Moreover, $W$ corresponds to a unique pair $(X, Y) \in \mathbf{W}_{n-i}^{M_1} \times \mathbf{W}_{i}^{M_2}$, since $X_1 Y_1 = X_2 Y_2$ in $M$ if and only if $X_1 = X_2$ in $M_1$ and $Y_1 \equiv Y_2$ (which is the same as $Y_1 = Y_2 \neq 0$ in $M_2$). Thus, we have
\[
|\{W \in \mathbf{W}_n^M \mid d(W) = i\}| = |\mathbf{W}_{n-i}^{M_1}| |\mathbf{W}_{i}^{M_2}|.
\]
Now we obtain
\[
|\mathbf{W}_n^M| = \sum_{i=0}^n |\{W \in \mathbf{W}_n^M \mid d(W) = i\}| = \sum_{i=0}^n |\mathbf{W}_{n-i}^{M_1}| |\mathbf{W}_{i}^{M_2}|,
\]
as desired.
\end{proof}

\subsection{Multiplication of Rank-Generating Functions}

In this subsection, we convert the definitions and results in \Cref{subsec: conv of monoids} to the poset version, study the corresponding rank-generating functions, and prove \Cref{thm:upho times semi-upho}.

\begin{definition}\label{def:poset convolution}
    Given a colored upho poset $\tilde{P}$ and a colored tree-like semi-upho poset $\tilde{S}$, the \emph{$\bx$-convolution} $\tilde{P}\rtimes_{\bx}\tilde{S}$ is defined by the $\bx$-convolution of their corresponding monoids. Explicitly, $\tilde{P}\rtimes_{\bx}\tilde{S}\coloneqq \mathcal{\tilde{P}}(\mathcal{M}(\tilde{P})\rtimes_{\bx}\mathcal{M}_0(\tilde{S}))$.
\end{definition}

\begin{example}
\Cref{color} depicts the $\bx$-convolution of the colored posets corresponding to the monoids in \Cref{eg:eg for fig upho times semiupho}. In this figure, the black edges corresponds to $x$, the red edges correspond to $y_1$, and the blue edges correspond to $y_2$.
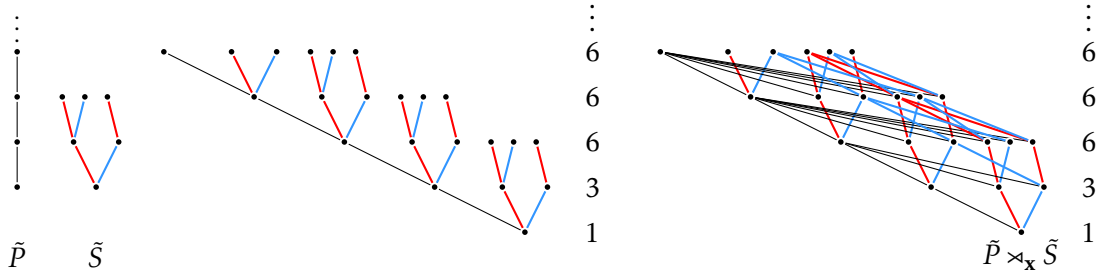
\begin{figure}[H]
\centering
\input{Figure/bigwhole}
\caption{Construction of $\tilde{P}\rtimes_{\bx}\tilde{S}$. }
\label{color}
\end{figure}
\end{example}

Now we are ready to prove \Cref{thm:upho times semi-upho}.

\begin{proof}[Proof of \Cref{thm:upho times semi-upho}]
    Since $f(x)$ is a regular upho function, there exists a colored upho poset $\tilde{P}$ such that $f(x) = F_{P}$. Likewise, since $g(x)$ is a regular semi-upho function, by \Cref{prop:reg semiupho= tree gen} and \Cref{lemma: trees are regular}, there exists a colored tree-like semi-upho poset $\tilde{S}$ such that $g(x) = F_{S}$. By \Cref{thm: LCIF monoid convolution}, after fixing an arbitrary mapping $\bx: \mathcal{I}_{\mathcal{M}_0(\tilde{S})} \rightarrow \mathcal{I}_{\mathcal{M}(\tilde{P})}$, the convolution $\tilde{P} \rtimes_{\bx} \tilde{S}$ is a colored upho poset. Moreover, by \Cref{prop: enum of conv monoid}, we have $F_{\mathfrak{F}({\tilde{P} \rtimes_{\bx} \tilde{S}})} = F_{P} F_{S} = f(x) g(x)$, which implies that $f(x) g(x)$ is a regular upho function, as desired.
\end{proof}

Moreover, we have the following two practical corollaries.

\begin{cor}\label{cor:reg upho times reg upho}
        Let  $f(x), g(x)\in 1+x\Z_{\ge 0}[[x]]$ be regular upho functions, then $f(x)g(x)$ is regular upho.
\end{cor}

\begin{proof}
    This follows directly from \Cref{thm:upho times semi-upho} and the fact that any regular upho function is a regular semi-upho function.
\end{proof}

\begin{cor}\label{cor:upho times log-concave}
    Let $f(x)\in 1+x\Z_{\ge 0}[[x]]$ be a regular upho function, and $g(x)\in 1+x\Z_{\ge 0}[[x]]$ be a log-concave formal power series, then $f(x)g(x)$ is regular upho.
\end{cor}

\begin{proof}
    This follows directly from \Cref{thm: log-concave rgf of semi-upho} and \Cref{thm:upho times semi-upho}.
\end{proof}

\section{Totally Positive Upho Functions}\label{sec:6}
In this section, our goal is to prove \Cref{thm:main} and \Cref{conj:Gao1}. Since \Cref{lemma: Ep of upho symmetric} and \Cref{prop:Schur positive} imply that \Cref{conj:Gao1} follows as a relatively straightforward corollary of \Cref{thm:main}, our main effort is focused on proving \Cref{thm:main}, all totally positive formal power series are regular upho functions.

We outline the proof of \Cref{thm:main} as follows.

By \Cref{thm:working def of TP}, it suffices to show that any formal power series $f(x) \in 1 + x \mathbb{Z}_{\ge 0}[[x]]$ of the form $\frac{g(x)}{h(x)}$ is a regular upho function, where $g(x), h(x) \in 1 + x \mathbb{Z}[x]$ are polynomials with all roots of $g(x)$ real and negative, and all roots of $h(x)$ real and positive.

For the numerator, we have:

\begin{lemma}\label{lemma: numerator g(x)}
    If all roots of $g(x) \in 1 + x \mathbb{Z}[x]$ are real and negative, then $g(x)$ is log-concave.
\end{lemma}

\begin{proof}
    Write $g(x) = (1 + \mu_1 x)(1 + \mu_2 x) \cdots (1 + \mu_n x)$. Since each factor $1 + \mu_i x$ is log-concave, \Cref{lemma:logconcavetimeslogconcave} implies that $g(x)$ is log-concave, as required.
\end{proof}

Applying \Cref{cor:upho times log-concave}, it suffices to show that any formal power series $f(x) \in 1 + x \mathbb{Z}_{\ge 0}[[x]]$ of the form $\frac{1}{h(x)}$ is a regular upho function, where $h(x) \in 1 + x \mathbb{Z}[x]$ has all roots real and positive.

Using \Cref{cor:reg upho times reg upho}, we can furthermore restrict to the case where $h(x)$ is irreducible over $\mathbb{Z}$, as we can factor $h(x)$ into irreducibles in $\mathbb{Z}[x]$.

Since each $\frac{1}{1 - ax}$ ($a \in \mathbb{Z}_{>0}$) is the rank-generating function of an $a$-ary tree and hence a regular upho function, it suffices to consider $h(x)$ irreducible over $\mathbb{Z}[x]$ with degree at least $2$.

Factor $h(x)$ over $\mathbb{R}$:
\[
f(x) = \frac{1}{h(x)} = \prod_{i=1}^{n} \frac{1}{1 - \lambda_i x}, \quad 0 < \lambda_1 < \lambda_2 < \cdots < \lambda_n,
\]
where $h(x)$ is irreducible over $\mathbb{Z}[x]$ and $\deg h(x) \ge 2$. All $\lambda_i$ are distinct since any algebraic extension of $\mathbb{Q}[x]$ is separable. Since the product of all $-\lambda_i$ is the integer coefficient of $x^n$ in $h(x)$, there exists at least one $\lambda_i > 1$. This reduces the proof to two cases, termed \emph{type I unitary} and \emph{type II unitary}.

\begin{prop}\label{prop:type I}
    If a formal power series $f(x) \in 1 + x \mathbb{Z}_{\ge 0}[[x]]$ is of the form 
    \[
    f(x) = \frac{1}{h(x)} = \prod_{i=1}^{n} \frac{1}{1 - \lambda_i x}, \quad 0 < \lambda_1 < \cdots < \lambda_n,
    \]
    where $h(x)$ is irreducible over $\mathbb{Z}[x]$, $\deg h(x) \ge 2$, and $1 < \lambda_{n-1} < \lambda_{n}$, then $f(x)$ is an upho function.
\end{prop}

\begin{prop} \label{prop:type II}
    If a formal power series $f(x) \in 1 + x \mathbb{Z}_{\ge 0}[[x]]$ is of the form
 \[
 f(x) = \frac{1}{h(x)} = \prod_{i=1}^{n} \frac{1}{1 - \lambda_i x}, \quad 0 < \lambda_1 < \cdots < \lambda_n,
 \]
 where $h(x)$ is irreducible over $\mathbb{Z}[x]$, $\deg h(x) \ge 2$, and $\lambda_{n-1} < 1 < \lambda_{n}$, then $f(x)$ is an upho function.
\end{prop}

The previous discussion shows that \Cref{prop:type I} and \Cref{prop:type II} imply \Cref{thm:main}. In the following two subsections, we handle these cases individually.

\subsection{Type I Unitary Totally Positive Functions}

In this subsection, our goal is to prove \Cref{prop:type I}. We state a technical lemma we use first.

\begin{lemma}\label{lemma:1-x/(1-ax)(1-bx)logconcave}
    If a formal power series $p(x)\in \mathbb{R}[[x]]$ is of the form \[p(x)=\sum_{i \ge 0}c_ix^i=\frac{1-x}{(1-\lambda_1 x)(1-\lambda_2 x)},\] and $\lambda_1, \lambda_2 \in \mathbb{R}$ satisfies $1<\lambda_1 <\lambda_2$, then $p(x)$ is log-concave and all coefficients $c_i$ are positive.
\end{lemma}

\begin{proof}[Proof of \Cref{prop:type I} assuming \Cref{lemma:1-x/(1-ax)(1-bx)logconcave}]
Note that
\[
f(x) = \frac{1}{h(x)} = \frac{1}{1 - x} \cdot \frac{1 - x}{(1 - \lambda_{n-1} x)(1 - \lambda_{n} x)} \cdot \prod_{i=1}^{n-2} \frac{1}{1 - \lambda_i x}.
\]
By \Cref{lemma:1-x/(1-ax)(1-bx)logconcave}, $\frac{1 - x}{(1 - \lambda_{n-1} x)(1 - \lambda_{n} x)}$ is log-concave with positive coefficients. Moreover, each term $\frac{1}{1 - \lambda_i x} = 1 + \lambda_i x + \lambda_i^2 x^2 + \cdots$ is also log-concave with positive coefficients. Thus, by \Cref{lemma:logconcavetimeslogconcave}, $(1 - x)f(x) = \frac{1 - x}{(1 - \lambda_{n-1} x)(1 - \lambda_{n} x)} \cdot \prod_{i=1}^{n-2} \frac{1}{1 - \lambda_i x}$ is log-concave. Moreover, $(1 - x)f(x)$ has integer coefficients since $f(x) \in 1 + x \mathbb{Z}[[x]]$. Additionally, as shown in \Cref{example:easyexamples}, $\frac{1}{1 - x}$ is a regular upho function. Therefore, by \Cref{cor:upho times log-concave}, we conclude that $f(x)=\frac{1}{1 - x}\cdot (1 - x)f(x)$ is a regular upho function.
\end{proof}

Finally, we prove \Cref{lemma:1-x/(1-ax)(1-bx)logconcave} by direct calculation.

\begin{proof}[Proof of \Cref{lemma:1-x/(1-ax)(1-bx)logconcave}]
By expanding $\sum_{i \ge 0}c_i x^i = \frac{1 - x}{(1 - \lambda_1 x)(1 - \lambda_2 x)}$, we have for each $i \ge 0$,
\[
c_i = \frac{1}{\lambda_1 - \lambda_2} \left( \lambda_1^{i+1} - \lambda_2^{i+1} - \lambda_1^i + \lambda_2^i \right).
\]

Note that for any $i \ge 2$,
\begin{align*}
    c_{i-1}^2 - c_{i-2}c_i &= \frac{1}{(\lambda_1 - \lambda_2)^2} \left[ (\lambda_1^i - \lambda_2^i - \lambda_1^{i-1} + \lambda_2^{i-1})^2 - (\lambda_1^{i+1} - \lambda_2^{i+1} - \lambda_1^i + \lambda_2^i)(\lambda_1^{i-1} - \lambda_2^{i-1} - \lambda_1^{i-2} + \lambda_2^{i-2}) \right] \\
    &= \frac{1}{(\lambda_1 - \lambda_2)^2} \Big[ -2\lambda_1^i \lambda_2^i + 2\lambda_1^i \lambda_2^{i-1} + 2\lambda_1^{i-1} \lambda_2^i - 2\lambda_1^{i-1} \lambda_2^{i-1} \\
    &\quad + \lambda_1^{i+1} \lambda_2^{i-1} - \lambda_1^{i+1} \lambda_2^{i-2} + \lambda_1^{i-1} \lambda_2^{i+1} - \lambda_1^{i-2} \lambda_2^{i+1} - \lambda_1^i \lambda_2^{i-1} + \lambda_1^i \lambda_2^{i-2} - \lambda_1^{i-1} \lambda_2^i + \lambda_1^{i-2} \lambda_2^i \Big] \\
    &= \lambda_1^{i-2} \lambda_2^{i-2} (\lambda_1 - 1)(\lambda_2 - 1) \ge 0.
\end{align*}

Hence, $\{c_i\}_{i \ge 0}$ is log-concave.

To prove each $c_i$ is positive, it suffices to show $\lambda_2^{i+1} - \lambda_1^{i+1} - \lambda_2^i + \lambda_1^i > 0$. Let $q(x) \coloneqq x^{i+1} - x^i$ with $x > 0$. We have $q'(x) = x^{i-1} ((i+1)x - i)$. Thus, $q'(x) > 0$ if and only if $x > \frac{i}{i+1}$. Since $\frac{i}{i+1} < 1 < \lambda_1 < \lambda_2$, we obtain
\[
\lambda_2^{i+1} - \lambda_1^{i+1} - \lambda_2^i + \lambda_1^i = q(\lambda_2) - q(\lambda_1) > 0,
\]
as desired.
\end{proof}

\subsection{Type II Unitary Totally Positive Functions}\label{subsec: const of type 2 num 1 poset}
In this subsection, our goal is to prove \Cref{prop:type II}. We prove by constructing a head-changing monoid (see \Cref{def:head-changing}) $M$ explicitly such that $F_{\mathcal{\tilde{P}}(M)}=f(x)$. Our construction relies on a surprising linear algebra result that we discovered, which appears to emerge from the void.

\begin{Miracle} \label{miracle}
    If a formal power series $f(x)=\sum_{i \ge 0}c_ix^i \in 1 + x \mathbb{Z}_{\ge 0}[[x]]$ is of the form
 \[
 f(x) = \frac{1}{h(x)} = \prod_{i=1}^{n} \frac{1}{1 - \lambda_i x}, \quad 0 < \lambda_1 < \cdots < \lambda_{n-1} < 1 < \lambda_{n},
 \]
 where $h(x)\in \mathbb{Z}[x]$ and $\deg h(x) \ge 2$, then there exist integers $l_1-1\ge l_2\ge \dots \ge l_n\ge 1$ such that
    \[c_i=
    \begin{pmatrix}
    1&1&1&\cdots&1
    \end{pmatrix}
    (L_n)^i
    \begin{pmatrix}
        1\\
        0\\
        0\\
        \vdots\\
        0
    \end{pmatrix}\text{ for any $i\ge 0$, where }L_n\coloneqq\begin{pmatrix}
    l_1 & l_2 & l_3 &\cdots&l_n\\
    1&1&0&\cdots&0\\
    1&1&1&\cdots&0\\
    \vdots&\vdots&\vdots&\ddots&\vdots\\
    1&1&1&\cdots&1
    \end{pmatrix} \in M_{n\times n}(\Z).
    \]
\end{Miracle}

\begin{notation}
We use $A(i,j)$ to denote the entry in the $i$-th row and $j$-th column of a matrix $A$.
\end{notation}

\begin{proof}[Proof of \Cref{prop:type II} assuming \Cref{miracle}]
   By \Cref{cor:corrsp for N-graded upho}, it suffices to construct an LCH monoid $M$ with $|\mathbf{W}_i^M| = c_i$ for any integer $i \ge 1$.

Let $l_1, l_2, \dots, l_n$ be the integers defined in \Cref{miracle}. Define $M \coloneqq \langle \mathbf{X} \mid R \rangle$, where $\mathbf{X}$ is an ordered set
\[
\mathbf{X} \coloneqq \{x_1^1< x_2^1< \dots< x_{l_1}^1< x_1^2< x_1^3< \dots< x_1^n\},
\]
and all defining relations in $R$ are
\begin{align*}
    &x^k_1 x^1_{t} = x_1^1 x^1_{t}, \quad 2 \le k \le n, \ 1 \le t \le l_1 - l_k; \\
    &x^k_1 x^j_1 = x^1_1 x^j_1, \quad 2 \le j < k \le n.
\end{align*}
Since $l_1 - 1 \ge l_2 \ge \dots \ge l_n$ by \Cref{miracle}, we have $l_{1} - l_k \ge 1$ for any $2 \le k \le n$, ensuring that all these defining relations are well-defined. Furthermore, since $M$ is a head-changing monoid, by \Cref{lemma:head-changing upho}, $M$ is an LCH monoid. In the remainder of the proof, we prove $|\mathbf{W}_i^M| = c_i$  for any $i\ge 0$.

Since $M$ is a head-changing monoid, the rightmost letter of any word representation of a given element in $M$ is fixed, which we refer to as the \emph{tail} of the element. For any integer $i \ge 1$, define $m_i^t$ as the number of distinct elements in $\mathbf{W}_i^M$ with the tail $x_s^t$ for some integer $s$. For $i \le 0$, set all $m_i^t = 0$, except for $m_0^{1} = 1$. Clearly, we have $|\mathbf{W}_i^M| = \sum_{t=1}^n m_i^t$. Thus, by \Cref{miracle}, it suffices to show that for any $i \ge 1$,
\[
\begin{pmatrix}
    m_{i}^1\\
    m_{i}^2\\
    m_{i}^3\\
    \vdots\\
    m_{i}^n
\end{pmatrix} = (L_n)^i
\begin{pmatrix}
    1\\
    0\\
    0\\
    \vdots\\
    0
\end{pmatrix},\text{ or equivalently, }
\begin{pmatrix}
    m_{i}^1\\
    m_{i}^2\\
    m_{i}^3\\
    \vdots\\
    m_{i}^n
\end{pmatrix}
= \begin{pmatrix}
    l_1 & l_2 & l_3 &\cdots&l_n\\
    1&1&0&\cdots&0\\
    1&1&1&\cdots&0\\
    \vdots&\vdots&\vdots&\ddots&\vdots\\
    1&1&1&\cdots&1
    \end{pmatrix}
\begin{pmatrix}
    m_{i-1}^1\\
    m_{i-1}^2\\
    m_{i-1}^3\\
    \vdots\\
    m_{i-1}^n
\end{pmatrix}.
\]
    When $i=1$, the last equality is trivially true by the definition of $\mathbf{X}$. Hence, it suffice to prove $m_i^t=\sum_{k=1}^nL_n(t,k)\cdot m_{i-1}^k$ for any $i\ge 2$ and $1\le t\le n$. 

    Denote $\mathbf{U}_i^t$ as the set of all words of length $i$ in $\mathbf{F}_i(\mathbf{X})$ that serve as the minimal word representation, under the $i$-lexicographic order, of some element in $\mathbf{W}_i^M$ with the tail $x_s^t$ for some integer $s$. Since distinct elements in $M$ have distinct minimal word representations, we have $m_j^t = |\mathbf{U}_j^t|$ for any $j \ge 1$. 
    
    Fix a word $W'x^k_s \in \mathbf{U}_{i-1}^k$; we claim that a word $W$ is in $\mathbf{U}_i^t$ with the prefix $W'x^t_s$ if and only if it is of the form $W'x^t_s x^k_{r}$, where $x^t_s x^k_{r} = x^1_1 x^k_{r}$ is not a defining relation in $R$.

Assuming our claim is true, by the definition of $\mathbf{X}$ and $R$, there are exactly $L_n(t, k)$ letters $x^k_r$ in $\mathbf{X}$ such that $x^t_s x^k_{r} = x^1_1 x^k_{r}$ is not a defining relation in $R$. Therefore, our claim implies that there are exactly $L_n(t, k)$ words in $\mathbf{U}_i^t$ with the prefix $W'x^k_s$. Thus, we have
\[
m_i^t = |\mathbf{U}_i^t| = \sum_{k=1}^n \sum_{W'x^k_s \in \mathbf{U}_{i-1}^{k}} |\{W \in \mathbf{U}_i^t \mid W \text{ has the prefix } W'x^k_s\}| = \sum_{k=1}^n \sum_{W'x^k_s \in \mathbf{U}_{i-1}^{k}} L_n(t, k) = \sum_{k=1}^n m_{i-1}^k \cdot L_n(t, k),
\]
proving the result we need.

Finally, we prove our claim.

For the ``if" part, suppose otherwise that $W'x^t_s x^k_{r}$ is not a minimal word representation. Since $x^t_s x^k_{r} = x^1_1 x^k_{r}$ is not a defining relation in $R$, by the definition of $R$, any defining relation does not involve the last letter $x^k_{r}$. Hence, by \Cref{prop:monoidtrans}, if $W'x^t_s x^k_{r}$ has a smaller word representation, then $W'x^t_s$ would also have a smaller word representation through the same transition, contradicting the minimality of $W'x^t_s$.

For the ``only if" part, suppose it fails; then $W'x^1_1 x^k_{r} = W'x^t_s x^k_{r}$ gives a smaller word representation, which is also a contradiction.

This completes the proof of the claim, and hence the proof of the proposition.
\end{proof}

In the rest of this subsection, we prove \Cref{miracle} by a series of technical lemmas.

\begin{lemma}\label{lem: c_i eq1}
    Let $h(x)=1+\sum_{i=1}^n (-1)^i h_i x^i$. Expand $f(x)=\frac{1}{h(x)}$ as $1+\sum_{i=1}^\infty c_i x^i$, then we have \[c_i=
    \begin{pmatrix}
    1&0&0&\cdots&0
    \end{pmatrix}
    (H_n)^i
    \begin{pmatrix}
        1\\
        0\\
        0\\
        \vdots\\
        0
    \end{pmatrix},\text{ where }H_n=\begin{pmatrix}
    h_1 & -h_2 & h_3 &\cdots&(-1)^{n}h_{n-1}&(-1)^{n+1}h_n\\
    1&0&0&\cdots&0&0\\
    0&1&0&\cdots&0&0\\
    \vdots&\vdots&\vdots&\ddots&\vdots&\vdots\\
    0&0&0&\cdots&0&0\\
    0&0&0&\cdots&1&0
    \end{pmatrix} \in M_{n\times n}(\Z)\]is the companion matrix of $h(x)$.
\end{lemma}
\begin{proof}
Let $c_0 = 1$ and $c_i = 0$ for $i < 0$. Since $h(x) f(x) = 1$, we have $
c_i = \sum_{j=1}^{n} (-1)^{j-1} c_{i-j} h_j$, yielding
\[
\begin{pmatrix}
    c_i \\
    c_{i-1} \\
    c_{i-2} \\
    \vdots \\
    c_{i-n+1}
\end{pmatrix}
= H_n
\begin{pmatrix}
    c_{i-1} \\
    c_{i-2} \\
    c_{i-3} \\
    \vdots \\
    c_{i-n}
\end{pmatrix}
\]
for any $i \ge 1$. Therefore, we obtain
\[
c_i =
\begin{pmatrix}
    1 & 0 & 0 & \cdots & 0
\end{pmatrix}
\begin{pmatrix}
    c_i \\
    c_{i-1} \\
    c_{i-2} \\
    \vdots \\
    c_{i-n+1}
\end{pmatrix}
= \begin{pmatrix}
        1 & 0 & 0 & \cdots & 0
    \end{pmatrix}
    (H_n)^i
    \begin{pmatrix}
        c_0 \\
        c_{-1} \\
        c_{-2} \\
        \vdots \\
        c_{-n+1}
    \end{pmatrix} = \begin{pmatrix}
        1 & 0 & 0 & \cdots & 0
    \end{pmatrix}
    (H_n)^i
    \begin{pmatrix}
        1 \\
        0 \\
        0 \\
        \vdots \\
        0
    \end{pmatrix}.
\]
\end{proof}

\begin{lemma} \label{lem: c_i eq2}
    In \Cref{lem: c_i eq1}, there exist integers $l_1, l_2,\dots, l_n\in \Z$ such that
    \[c_i=
    \begin{pmatrix}
    1&1&1&\cdots&1
    \end{pmatrix}
    (L_n)^i
    \begin{pmatrix}
        1\\
        0\\
        0\\
        \vdots\\
        0
    \end{pmatrix},\quad \text{where }L_n\coloneqq\begin{pmatrix}
    l_1 & l_2 & l_3 &\cdots&l_n\\
    1&1&0&\cdots&0\\
    1&1&1&\cdots&0\\
    \vdots&\vdots&\vdots&\ddots&\vdots\\
    1&1&1&\cdots&1
    \end{pmatrix} \in M_{n\times n}(\Z).
    \]
\end{lemma}
\begin{proof}
    Let $A_n, B_n$ be $n\times n$ matrices with
    \begin{align*}
        &A_n(i,j)=\left\{ 
    \begin{array}{ll}
        1,& \text{ if }i=1; \\
        0,& \text{ if }i>1 \text{ and } i+j<n+2;\\
        (-1)^{n-j}
        \begin{pmatrix}
            i-2\\
            n-j
        \end{pmatrix}, & \text{ if }i>1 \text{ and } i+j\geq n+2.
    \end{array}
    \right. \\
    &B_n(i,j)=\left\{ 
        \begin{array}{ll}
            (-1)^{j+1}
            \begin{pmatrix}
                n-1\\
                j-1
            \end{pmatrix}, & \text{ if }i=1; \\
            0,& \text{ if }i>1 \text{ and } (j=1\text{ or } i+j>n+2);\\
            (-1)^{j}
            \begin{pmatrix}
                n-i\\
                j-2
            \end{pmatrix}, & \text{ if }i>1 \text{ and } (j>1 \text{ and } i+j\leq n+2).
        \end{array}
    \right.
    \end{align*}
    Then $A_nB_n(1,1)=1$. For $i=1, j>1$, we have
    \[
        A_nB_n(1,j) =\sum_{k=1}^n A_n(1,k)B_n(k,j)= (-1)^{j+1}
            \begin{pmatrix}
                n-1\\
                j-1
            \end{pmatrix}
        +\sum_{k=2}^{n+2-j} (-1)^j
        \begin{pmatrix}
                n-k\\
                j-2
        \end{pmatrix}=0.
    \]
    
    For $i>1, 2\le j\le i$, we have
    \begin{align*}
        A_nB_n(i,j)=&\sum_{k=1}^n A_n(i,k)B_n(k,j)=\sum_{k=n+2-i}^{n+2-j} (-1)^{n-k}
        \begin{pmatrix}
            i-2\\
            n-k
        \end{pmatrix}\cdot
    (-1)^{j}
        \begin{pmatrix}
            n-k\\
            j-2
        \end{pmatrix}\\
     =&\sum_{k'=0}^{i-j} (-1)^{j-i+k'}
     \begin{pmatrix}
         i-j\\
         k'
     \end{pmatrix}
     \begin{pmatrix}
         i-2\\
         j-2
     \end{pmatrix}=\left\{\begin{array}{ll}
         0, & j<i; \\
         1, & j=i.
     \end{array}\right. \quad (\text{let } k'= k-(n+2-i))\\
    \end{align*}
    For $1<i<j$, since for any $1\le k\le n$, at least one of the conditions  $i+k<n+2$ or $k+j>n+2$ holds, we have 
    \[A_nB_n(i,j)=\sum_{k=1}^n A_n(i,k)B_n(k,j)=0.\]
    
    In conclusion, we have $A_nB_n=I_n$, the $n\times n$ identity matrix. Now let $L_n\coloneqq B_n H_n A_n$, then we have 
    \begin{align*}
      c_i=&
    \begin{pmatrix}
    1&0&0&\cdots&0
    \end{pmatrix}
    (H_n)^i
    \begin{pmatrix}
        1\\
        0\\
        0\\
        \vdots\\
        0
    \end{pmatrix}
    = \begin{pmatrix}
    1&0&0&\cdots&0
    \end{pmatrix}
    ((B_n)^{-1} L_n (A_n)^{-1})^i
    \begin{pmatrix}
        1\\
        0\\
        0\\
        \vdots\\
        0
    \end{pmatrix}\\
    =& \begin{pmatrix}
    1&0&0&\cdots&0
    \end{pmatrix}
    (A_n L_n B_n)^i
    \begin{pmatrix}
        1\\
        0\\
        0\\
        \vdots\\
        0
    \end{pmatrix}
    = \begin{pmatrix}
    1&0&0&\cdots&0
    \end{pmatrix}
    A_n (L_n)^i B_n
    \begin{pmatrix}
        1\\
        0\\
        0\\
        \vdots\\
        0
    \end{pmatrix}\\
    =& \begin{pmatrix}
    1&1&1&\cdots&1
    \end{pmatrix}
    (L_n)^i
    \begin{pmatrix}
        1\\
        0\\
        0\\
        \vdots\\
        0
    \end{pmatrix},
    \end{align*} 
    
       Finally, we show that $L_n$ is of the desired form. Clearly $l_1, l_2,\dots, l_n\in \Z$ by definition. Moreover, for $i\ge 2$, it is straightforward to check that \[L_n(i,j)=\sum_{r=1}^n\sum_{s=1}^n B_n(i,r)H_n(r,s)A_n(s,j)=\sum_{s=1}^{n-1} B_n(i,s+1)A_n(s,j)=\left\{\begin{array}{ll}
        1,  &  i\ge 2,\  i\ge j;\\
        0,  &  j> i\ge 2.
     \end{array}\right.\]
\end{proof}

\begin{lemma}\label{lemma:l_i}
    The integers $l_1,l_2,\dots,l_n$ in \Cref{lem: c_i eq2} satisfy \[l_i= h_1 + \sum_{t=0}^{i-2} (-1)^{t+1}{n-i+t \choose n-i}\cdot h_{n-i+2+t} -n +i -(1-\delta_1(i)),\text{ where }\delta_1(i)=\left\{
        \begin{array}{ll}
            1,&\text{ if }i=1;  \\
            0,&\text{ otherwise.}
        \end{array}\right.\]
\end{lemma}
\begin{proof}
Note that
        \begin{align*}
    l_i =L_n(1, i) =& \sum_{r=1}^n\sum_{s=1}^n B_n(1,r) H_n(r,s) A_n(s,i)\\
        =& \sum_{r=1}^n B_n(1,r)H_n(r,1) + \sum_{r=1}^n \sum_{s=2}^n B_n(1,r) H_n(r,s) A_n(s,i)\\
        =& \sum_{r=1}^n B_n(1,r)H_n(r,1) + \sum_{s=2}^n H_n(1,s)A_n(s,i) +\sum_{s=2}^n B_n(1,s+1) H_n(s+1,s) A_n(s,i)\\
        =& h_1 + (-1){n-1 \choose 1} + \sum_{s=2}^n (-1)^{s-1} h_s (-1)^{n-i}{s-2 \choose n-i} +\sum_{s=2}^n (-1)^s{n-1 \choose s}(-1)^{n-i}{s-2 \choose n-i}\\
        =& h_1 + (-1){n-1 \choose 1} + \sum_{s=n-i+2}^n(-1)^{n-i+s-1}{s-2 \choose n-i} \cdot h_s +\sum_{s=2}^n (-1)^s{n-1 \choose s}(-1)^{n-i}{s-2 \choose n-i}\\
        =& h_1 + (-1){n-1 \choose 1} + \sum_{t=0}^{i-2} (-1)^{t+1}{n-i+t \choose n-i}\cdot h_{n-i+2+t}  +\sum_{s=2}^n (-1)^s{n-1 \choose s}(-1)^{n-i}{s-2 \choose n-i}.
    \end{align*} 
    
    Hence, it suffices to prove that
    \[\sum_{s=2}^n (-1)^s{n-1 \choose s}(-1)^{n-i}{s-2 \choose n-i}
        =i-2+\delta_1(i).\]
    Let $k(i)\coloneqq \sum_{s=2}^n (-1)^s{n-1 \choose s}(-1)^{n-i}{s-2 \choose n-i}$, where $1\le i\le n $. Clearly $k(i)=0$ for $i=1,2$. For $1\le i\le n-2$, we have:
        \begin{align*}
            k(i+1)-k(i)&=\sum_{s=2}^n (-1)^s{n-1 \choose s}\left((-1)^{n-i-1}{s-2 \choose n-i-1}-(-1)^{n-i}{s-2 \choose n-i}\right)\\
            &=\sum_{s=2}^n (-1)^s{n-1 \choose s}(-1)^{n-i-1}{s-1 \choose n-i},\\
            k(i+2)-2k(i+1)+k(i)&=\sum_{s=2}^n (-1)^s{n-1 \choose s}\left((-1)^{n-i-2}{s-1 \choose n-i-1}-(-1)^{n-i-1}{s-1 \choose n-i}\right)\\
            &=\sum_{s=2}^n (-1)^s{n-1 \choose s}(-1)^{n-i-2}{s \choose n-i}\\
            &=\sum_{s=2}^n (-1)^{n-i}{n-1 \choose n-i}(-1)^s{i-1 \choose n-1-s}\\
            &=(-1)^{n-i}{n-1 \choose n-i}\sum_{s'=-1}^{n-3} (-1)^{n-1-s'}{i-1 \choose s'} \quad (\text{let }s'=n-1-s) \\
            &=\left\{\begin{array}{cc}
                0, &  i>1;\\
                1, &  i=1.
            \end{array}\right.
        \end{align*}
        Hence, $k(i) = i - 2$ for $3 \le i \le n$. Therefore, $k(i) = i - 2 + \delta_1(i)$ for all $1 \le i \le n$.
\end{proof}

\begin{lemma}\label{lem: c_i eq3}
    In \Cref{lem: c_i eq2}, if $h(x)=\prod_{i=1}^n(1-\lambda_i x)$ such that $0 < \lambda_1 < \cdots < \lambda_{n-1} < 1 < \lambda_{n}$, then we have $l_1-1\ge l_2\ge \dots \ge l_n\ge 1$.
\end{lemma}

\begin{proof}
    First we show that $l_n\ge 1$. By \Cref{lemma:l_i}, we have
    \[
        l_n=h_1+(-1){n-1 \choose 1}+\sum_{s=2}^n(-1)^{s-1}h_s+\sum_{s=2}^n (-1)^s{n-1 \choose s}=\sum_{s=1}^n(-1)^{s-1}h_s -1=-\prod_{j=1}^{n}(1-\lambda_j).
    \]
    Moreover, since $0<\lambda_1<\cdots< \lambda_{n-1}<1<\lambda_n$, we have $l_n>0$, and hence the integer $l_n \ge 1$.

    Now it suffices to show that $l_i - l_{i+1} \ge \delta_1(i)$ for all $1\le i \le n-1$. Again by \Cref{lemma:l_i}, we have
    \begin{align*}
        l_i=&h_1 + \sum_{t=0}^{i-2} (-1)^{t+1}{n-i+t \choose n-i}\cdot h_{n-i+2+t} -n +i -(1-\delta_1(i))\\
        =&h_1 + \sum_{t'=0}^{i-1} (-1)^{t'}{n-i-1+t' \choose n-i}\cdot h_{n-i+1+t'} -n +i -(1-\delta_1(i))\quad (\text{let }t'=t+1)
    \end{align*}
    and
\[l_{i+1}=h_1 + \sum_{t=0}^{i-1} (-1)^{t+1}{n-i-1+t \choose n-i-1}\cdot h_{n-i+1+t} -n +i. \quad(\text{as }\delta_1(i+1)=0)\]
    So
    \begin{align*}
        l_i-l_{i+1}=&\sum_{t=0}^{i-1}\left((-1)^t\left({n-i-1+t \choose n-i}+{n-i-1+t \choose n-i-1}\right)\cdot h_{n-i+1+t}\right) -(1-\delta_1(i))\\
        =&\sum_{t=0}^{i-1}\left((-1)^t{n-i+t \choose n-i}\cdot h_{n-i+1+t}\right) -(1-\delta_1(i))\\
        =&\sum_{t=0}^{i-1}\left((-1)^t{n-i+t \choose n-i}\sum_{\substack{J \subseteq [n], \\ |J|=n-i+1+t}}\left(\prod_{j\in J} \lambda_j\right)\right) -(1-\delta_1(i))\\
        =&\sum_{s=0}^{i-1}\left( \lambda_{n-s}\left(\sum_{t=0}^{i-s-1}\left((-1)^t{n-i+t \choose n-i}\sum_{\substack{J \subseteq [n-s-1], \\ |J|=n-i+t}}\left(\prod_{j\in J} \lambda_j\right)\right) \right)\right) -(1-\delta_1(i))\\
        =&\sum_{s=0}^{i-1}\left( \lambda_{n-s}\left(\sum_{t=0}^{i-s-1}\left(\sum_{\substack{J \subseteq [n-s-1], \\ |J|=n-i+t}}\left( \sum_{\substack{R \subseteq J,\\ |R|=n-i}} \left( \prod_{r \in R} \lambda_r \cdot \prod_{k \in J\backslash R} (-\lambda_k)\right)\right)\right) \right)\right) -(1-\delta_1(i))\\
        =&\sum_{s=0}^{i-1}\left( \lambda_{n-s}\left( \sum_{\substack{R \subseteq [n-s-1],\\ |R|=n-i}} \left( \prod_{r \in R} \lambda_r \cdot \prod_{k \in [n-s-1]\backslash R} (1-\lambda_k)\right)\right)\right) -(1-\delta_1(i))\\
        >&-1+\delta_1(i).
    \end{align*}
    The last step follows from $0<\lambda_1<\cdots< \lambda_{n-1}<1<\lambda_n$. Since $l_i-l_{i+1} \in \Z$, we have $l_i-l_{i+1} \ge \delta_1(i)$.
\end{proof}

We can see from the following example how \Cref{lem: c_i eq2} and \Cref{lem: c_i eq3} work.
\begin{example}
For $n=5$, we have
\begin{align*}
    H_5&=\begin{pmatrix}
    h_1 & -h_2 & h_3 & -h_4 & h_5\\
    1&0&0&0&0\\
    0&1&0&0&0\\
    0&0&1&0&0\\
    0&0&0&1&0
    \end{pmatrix}\quad
    A_5=\begin{pmatrix}
    1&1&1&1&1\\
    0&0&0&0&1\\
    0&0&0&-1&1\\
    0&0&1&-2&1\\
    0&-1&3&-3&1
    \end{pmatrix}\quad
    B_5=(A_5)^{-1}= \begin{pmatrix}
    1&-4&6&-4&1\\
    0&1&-3&3&-1\\
    0&1&-2&1&0\\
    0&1&-1&0&0\\
    0&1&0&0&0
    \end{pmatrix}\\
    L_5&=B_5H_5A_5=
    \begin{pmatrix}
    1&-4&6&-4&1\\
    0&1&-3&3&-1\\
    0&1&-2&1&0\\
    0&1&-1&0&0\\
    0&1&0&0&0
    \end{pmatrix}
    \cdot
    \begin{pmatrix}
    h_1 & -h_2 & h_3 & -h_4 & h_5\\
    1&0&0&0&0\\
    0&1&0&0&0\\
    0&0&1&0&0\\
    0&0&0&1&0
    \end{pmatrix}
    \cdot
    \begin{pmatrix}
    1&1&1&1&1\\
    0&0&0&0&1\\
    0&0&0&-1&1\\
    0&0&1&-2&1\\
    0&-1&3&-3&1
    \end{pmatrix}=
    \begin{pmatrix}
    l_1 & l_2 & l_3 & l_4 & l_5\\
    1&1&0&0&0\\
    1&1&1&0&0\\
    1&1&1&1&0\\
    1&1&1&1&1
    \end{pmatrix}\\
    c_i&=
    \begin{pmatrix}
    1&0&0&0&0
    \end{pmatrix}
    \cdot
    \begin{pmatrix}
    h_1 & -h_2 & h_3 & -h_4 & h_5\\
    1&0&0&0&0\\
    0&1&0&0&0\\
    0&0&1&0&0\\
    0&0&0&1&0
    \end{pmatrix}^i
    \cdot
    \begin{pmatrix}
        1\\
        0\\
        0\\
        0\\
        0
    \end{pmatrix}\\
        &=
    \begin{pmatrix}
    1&0&0&0&0
    \end{pmatrix}
    \cdot
    \left(\begin{pmatrix}
    1&1&1&1&1\\
    0&0&0&0&1\\
    0&0&0&-1&1\\
    0&0&1&-2&1\\
    0&-1&3&-3&1
    \end{pmatrix}
    \cdot
    \begin{pmatrix}
    l_1 & l_2 & l_3 & l_4 & l_5\\
    1&1&0&0&0\\
    1&1&1&0&0\\
    1&1&1&1&0\\
    1&1&1&1&1
    \end{pmatrix}
    \cdot
     \begin{pmatrix}
    1&-4&6&-4&1\\
    0&1&-3&3&-1\\
    0&1&-2&1&0\\
    0&1&-1&0&0\\
    0&1&0&0&0
    \end{pmatrix}\right)^i
    \cdot
    \begin{pmatrix}
        1\\
        0\\
        0\\
        0\\
        0
    \end{pmatrix}\\
    &=
    \begin{pmatrix}
    1&0&0&0&0
    \end{pmatrix}
    \cdot
    \begin{pmatrix}
    1&1&1&1&1\\
    0&0&0&0&1\\
    0&0&0&-1&1\\
    0&0&1&-2&1\\
    0&-1&3&-3&1
    \end{pmatrix}
    \cdot
    \begin{pmatrix}
    l_1 & l_2 & l_3 & l_4 & l_5\\
    1&1&0&0&0\\
    1&1&1&0&0\\
    1&1&1&1&0\\
    1&1&1&1&1
    \end{pmatrix}^i
    \cdot
     \begin{pmatrix}
    1&-4&6&-4&1\\
    0&1&-3&3&-1\\
    0&1&-2&1&0\\
    0&1&-1&0&0\\
    0&1&0&0&0
    \end{pmatrix}
    \cdot
    \begin{pmatrix}
        1\\
        0\\
        0\\
        0\\
        0
    \end{pmatrix}\\
    &=
    \begin{pmatrix}
    1&1&1&1&1
    \end{pmatrix}
    \cdot
    \begin{pmatrix}
    l_1 & l_2 & l_3 & l_4 & l_5\\
    1&1&0&0&0\\
    1&1&1&0&0\\
    1&1&1&1&0\\
    1&1&1&1&1
    \end{pmatrix}^i
    \cdot
    \begin{pmatrix}
        1\\
        0\\
        0\\
        0\\
        0
    \end{pmatrix},\text{ where }\left.\begin{array}{ll}
        l_1=& h_1 -4\\
        l_2=& h_1 -h_5 -4\\
        l_3=& h_1 -h_4 +3h_5 -3\\
        l_4=& h_1 -h_3 +2h_4 -3h_5 -2\\
        l_5=& h_1 -h_2 +h_3 -h_4 +h_5 -1.
     \end{array}\right.
\end{align*}
Expand each $h_i$ in terms of $\lambda_i$, and note that each $l_j - l_{j+1}$ is an integer, we have
\begin{align*}
    l_1-l_2=& h_5= \lambda_5\lambda_1\lambda_2\lambda_3\lambda_4\ge 1\\
    l_2-l_3=& h_4 -4h_5 -1\\
        =& \lambda_5(\lambda_1\lambda_2\lambda_3(1-\lambda_4)+\lambda_1\lambda_2\lambda_4(1-\lambda_3)+\lambda_1\lambda_3\lambda_4(1-\lambda_2)+\lambda_2\lambda_3\lambda_4(1-\lambda_1))\\
        +&\lambda_4\lambda_1\lambda_2\lambda_3-1\ge 1-1=0\\
    l_3-l_4=& h_3 -3h_4 +6h_5 -1\\
        =& \lambda_5(\lambda_1\lambda_2(1-\lambda_3)(1-\lambda_4)+\lambda_1\lambda_3(1-\lambda_2)(1-\lambda_4)+\lambda_1\lambda_4(1-\lambda_2)(1-\lambda_3)\\
        &+\lambda_2\lambda_3(1-\lambda_1)(1-\lambda_4)+\lambda_2\lambda_4(1-\lambda_1)(1-\lambda_3)+\lambda_3\lambda_4(1-\lambda_1)(1-\lambda_2))\\
        +&\lambda_4(\lambda_1\lambda_2(1-\lambda_3)+\lambda_1\lambda_3(1-\lambda_2)+\lambda_2\lambda_3(1-\lambda_1))\\
        +&\lambda_3\lambda_1\lambda_2-1\ge 1-1=0\\
    l_4-l_5=& h_2 -2h_3 +3h_4 -4h_5 -1\\
        =& \lambda_5(\lambda_1(1-\lambda_2)(1-\lambda_3)(1-\lambda_4)+\lambda_2(1-\lambda_1)(1-\lambda_3)(1-\lambda_4)\\
        &+\lambda_3(1-\lambda_1)(1-\lambda_2)(1-\lambda_4)+\lambda_4(1-\lambda_1)(1-\lambda_2)(1-\lambda_3))\\
        +&\lambda_4(\lambda_1(1-\lambda_2)(1-\lambda_3)+\lambda_2(1-\lambda_1)(1-\lambda_3)+\lambda_3(1-\lambda_1)(1-\lambda_2))\\
        +&\lambda_3(\lambda_1(1-\lambda_2)+\lambda_2(1-\lambda_1))\\
        +&\lambda_2\lambda_1-1 \ge 1-1=0.
\end{align*}
\end{example}

\begin{proof}[Proof of \Cref{miracle}]
    This follows directly from  \Cref{lem: c_i eq2} and \Cref{lem: c_i eq3}.
\end{proof}

\subsection{Total Positivity Implies Regular Upho}

Now we are ready to proved our last main results, \Cref{thm:main} and \Cref{conj:Gao1}.

\begin{proof}[Proof of \Cref{thm:main}]
    As discussed at the beginning of \Cref{sec:6}, \Cref{prop:type I} and \Cref{prop:type II} together imply \Cref{thm:main}.
\end{proof}

\begin{proof}[Proof of \Cref{conj:Gao1}]
\Cref{lemma: Ep of upho symmetric} implies that the Ehrenborg quasi-symmetric function of an upho poset is always a symmetric function, and \Cref{prop:Schur positive} shows that they are exactly totally positive formal power series which are upho functions. Moreover, \Cref{thm:main} implies that all totally positive formal power series are upho functions, thereby proving the theorem.
\end{proof}

\section{Further Directions}\label{sec:7}

\subsection{Regularity of Finitary Upho Posets}
The main results in this paper are based on the monoid representation of upho posets. A fundamental question is: does any finitary upho poset have a monoid representation? We believe that \Cref{conj:upho regular} and \Cref{conj:semiupho regular} are central conjectures in the study of upho posets and semi-upho posets. As far as we know, even the finite type $\mathbb{N}$-graded cases of these conjectures remain open, and proving them would represent significant progress in the study of upho posets.

\begin{conj}\label{conj:n upho regular}
    All finite type $\N$-graded upho posets are regular.
\end{conj}

\begin{conj}\label{conj:n semiupho regular}
    All finite type $\N$-graded semi-upho posets are regular.
\end{conj}
Readers are referred to \cite[Section~3]{HopkinsLewis24} for further discussions on this problem.

\subsection{Upho Posets beyond Finitary}
As shown in the proofs of \Cref{lemma:poset to LCIF monoid} and \Cref{lemma:LCIF monoid to posets}, we have the following result.

\begin{lemma}
    A poset defined by the left-divisibility of a monoid is an upho poset if and only if the monoid is an LCIF monoid.
\end{lemma}

As noted in \Cref{remark: LCIF not atomic}, LCIF monoids are not necessarily atomic, so the corresponding upho poset may not be finitary. Moreover, the upho poset corresponding to the LCIF monoid in \Cref{remark: LCIF not atomic} is not even Noetherian. This raises the following question:

\begin{question}
    Characterize all upho posets defined by the left-divisibility of LCIF monoids.
\end{question}

\subsection{Recursive Characterization of Semi-Upho Functions}
In \Cref{subsec: semi-upho spanning tree} and \Cref{subsec: greedy 0 monoid}, we establish a framework for characterizing all regular semi-upho functions, reducing the problem to proving \Cref{conj:tree like=greedy series} and \Cref{conj:tree like recursive}. Furthermore, inspired by \Cref{lemma: regular semi-upho by tree}, we suggest that it may be possible to characterize all semi-upho functions without a complete proof of \Cref{conj:semiupho regular}.

\begin{conj}\label{conj: semi-upho by tree}
    Given a finite type finitary semi-upho poset $S$, there exists a tree-like semi-upho poset $T$ such that the Hasse diagram of $T$ is a spanning tree of the Hasse diagram of $S$. Moreover, if $S$ is $\mathbb{N}$-graded, then the rank-generating functions satisfy $F_{S} = F_{T}$.
\end{conj}

If \Cref{conj: semi-upho by tree} holds, then proofs of \Cref{conj:tree like=greedy series} and \Cref{conj:tree like recursive} would yield a complete characterization of all semi-upho functions.

\subsection{Recursive Characterization of Regular Upho Functions}\label{subsec: char of upho function}
Similar to \Cref{subsec: greedy 0 monoid}, we introduce the \emph{greedy LCH monoid series} to establish a recursive framework for characterizing all regular upho functions.

\begin{definition}
Given a formal power series $f(x) = 1 + c_1 x + c_2 x^2 + \cdots \in 1 + x \mathbb{Z}_{\ge 0}[[x]]$, the \emph{greedy LCH monoid series} $\{M_k^{f}\}_{k \ge 1}$ of $f(x)$ is a sequence of monoids defined recursively as follows:
\begin{itemize}
    \item Define $M_1^{f}$ as the free monoid generated by the ordered set $\mathbf{X} \coloneqq \{x_1< x_2< \dots < x_{c_1}\}$. 
    \item For $M_{k}^{f} = \langle \mathbf{X} \mid R_{k} \rangle$, if $c_k \le c_{k+1} \le |\mathbf{W}_{k+1}^{M_{k}^f}|$, then define $M_{k+1}^{f} \coloneqq \langle \mathbf{X} \mid R_{k+1} \rangle$, where $R_{k+1}$ extends $R_{k}$ with additional head-changing relations $x_i W = x_{i-1} W$ for the largest $|\mathbf{W}_{k+1}^{M_{k}^{f}}| - c_{k+1}$ elements $x_i W$ in $\mathbf{W}_{k+1}^{M_k^f}$ under the $(k+1)$-lexicographic order on $\mathbf{W}_{k+1}^{M_k^f}$ (see \Cref{def: k-lexi order of W^M_n}). The condition $c_{k} \le c_{k+1}$ ensures that each selected element $x_i W$ satisfies $i \ge 2$, making the relation $x_i W = x_{i-1} W$ well-defined.
\end{itemize}
We refer to $M_k^{f}$ as the \emph{$k$-th greedy LCH monoid} of $f(x)$.
\end{definition}

We omit the superscript $f$ when the power series is clear from context. Analogous to \Cref{prop:greedy 0-series implies semi-upho function} and \Cref{prop:semi-upho upper bound}, the following properties hold for regular upho functions.

\begin{prop}\label{prop: greedy LCH series implies upho function}
    Given a formal power series $f(x) \in 1 + x \mathbb{Z}_{\ge 0}[[x]]$, if $f(x)$ admits an infinite greedy LCH monoid series $\{M_k\}_{k\ge 1}$, then $f(x)$ is a regular upho function.
\end{prop}

\begin{prop}\label{prop:upho upper bound}
    Given a formal power series $f(x) = 1 + c_1 x + c_2 x^2 + \cdots \in 1 + x \mathbb{Z}_{\ge 0}[[x]]$, if for any integer $k \ge 1$, the $k$-th greedy LCH monoid $M_k$ is defined implies the inequality $c_k\le c_{k+1}\le |\mathbf{W}_k^{M_{k+1}}|$ holds, then $f(x)$ is a regular upho function.
\end{prop}

Moreover, we conjecture that \Cref{prop: greedy LCH series implies upho function} and \Cref{prop:upho upper bound} characterize all regular upho functions.

\begin{conj}\label{conj: upho greedy series}
        A formal power series $f(x)\in 1 + x \mathbb{Z}_{\ge 0}[[x]]$ is a regular upho function if and only if $f(x)$ admits an infinite greedy LCH monoid series $\{M_k\}_{k\ge 1}$.
\end{conj}

\begin{conj}\label{conj: upho recursive}
    A formal power series $f(x) = 1 + c_1 x + c_2 x^2 + \cdots \in 1 + x \mathbb{Z}_{\ge 0}[[x]]$ is a regular upho function if and only if for any integer $k \ge 2$, the $k$-th greedy LCH monoid $M_{k}^0$ is defined implies the inequality $c_k\le c_{k+1}\le |\mathbf{W}_k^{M_{k}}|$ holds.
\end{conj}

If \Cref{conj:upho regular} holds, then proving \Cref{conj: upho greedy series} and \Cref{conj: upho recursive} would furthermore yield a complete characterization of all upho functions.

\subsection{Characterizing Rational Upho Functions}
Proving \Cref{conj: upho greedy series} and \Cref{conj: upho recursive} would yield a complete characterization of (regular) upho functions. However, due to the intrinsic difficulty of counting elements in LCH monoids, a more accessible characterization may be essential for practical applications. While an elegant and comprehensive characterization of all upho functions is likely unattainable \cite[Theorem~1.4]{MR4318812}, many interesting upho posets in current research have rational rank-generating functions. For example, \cite[Corollary~6]{MR4432972} shows that the rank-generating function of any upho meet semilattice, and thus any upho lattice, is a \emph{reciprocal rational function}, which is the multiplicative inverse of some polynomial. Additionally, \cite[Theorem~1.3]{MR4318812} provides a complete characterization of the rank-generating functions of planar upho posets, which are also reciprocal rational functions.

Thus, the classification of rational upho functions is a significant problem in the study of upho functions.

As our \Cref{thm:main} shows that totally positive formal power series are upho functions, this is not an exhaustive list of rational upho functions. \Cref{example:feiposet} illustrates that for a rational upho function $\frac{g(x)}{h(x)}$, $g(x)$ may have real positive roots, and $h(x)$ may have real negative roots. Furthermore, by \Cref{cor:upho times log-concave},
\[
\frac{g(x)}{h(x)} = \frac{1 - x + x^2}{(1 - x)(1 - 4x + 4x^2)} = \frac{1}{1 - x} \cdot \frac{1 - x + x^2}{1 - 4x + 4x^2}
\]
is a regular upho function, as $\frac{1}{1 - x}$ is regular upho and $\frac{1 - x + x^2}{1 - 4x + 4x^2}$ is log-concave; but $g(x)$ have imaginary roots. Likewise, $h(x)$ can also have imaginary roots, as $\frac{1}{1 - 2x + x^3}$ is an upho function by \cite[Theorem~1.3]{MR4318812}. Thus, classifying rational upho functions remains a challenging problem.

\begin{Question}
    Classify rational (regular) upho functions.
\end{Question}

A slightly simpler problem is to characterize all reciprocal rational upho functions.

\begin{Question}
    Classify reciprocal rational (regular) upho functions.
\end{Question}

This question remains subtle, as \cite[Theorem~1.3]{MR4318812} implies that both $\frac{1}{1 - 2x + x^3}$ and $ \frac{1}{1 - 3x + x^2 + x^3}$ are upho functions, where the denominator may have imaginary or real negative roots.

\section*{Acknowledgement}
This research was conducted as part of the PACE program during the summer of 2023 at Peking University. We thank Prof. Yibo Gao for proposing the project. We also thank Yumou Fei, Samuel F. Hopkins, and Joel B. Lewis for valuable discussions. 

\bibliographystyle{amsplain}
\bibliography{ref}
\addcontentsline{toc}{section}{Bibliography}

\end{document}

%% file: Figure/tikzstyle.tex
\usepackage{tikz}
\usetikzlibrary{backgrounds}
\usetikzlibrary{arrows}
\usetikzlibrary{shapes,shapes.geometric,shapes.misc}

\tikzstyle{tikzfig}=[baseline=-0.25em,scale=0.5]

\pgfkeys{/tikz/tikzit fill/.initial=0}
\pgfkeys{/tikz/tikzit draw/.initial=0}
\pgfkeys{/tikz/tikzit shape/.initial=0}
\pgfkeys{/tikz/tikzit category/.initial=0}

\pgfdeclarelayer{edgelayer}
\pgfdeclarelayer{nodelayer}
\pgfsetlayers{background,edgelayer,nodelayer,main}

\tikzstyle{none}=[inner sep=0mm]

\tikzstyle{every loop}=[]

\tikzstyle{label}=[inner sep=0.15mm, font={\footnotesize}]
\tikzstyle{vert}=[fill=black, draw=white, shape=circle, line width=0.25mm, tikzit shape=circle, inner sep=0.35mm]
\tikzstyle{label-small}=[inner sep=0.1mm, font={\scriptsize}]
\tikzstyle{label-white}=[fill=white, draw=none, shape=circle, inner sep=0.07mm, font={\scriptsize}]
\tikzstyle{vert-white}=[fill=white, draw=black, shape=circle, inner sep=0.35 mm]
\tikzstyle{vert-red}=[fill=red, draw=white, shape=circle, inner sep=0.4 mm, font={\scriptsize}, text=black]

\tikzstyle{blue}=[-, draw={rgb,255: red,49; green,149; blue,255}, line width=0.8pt]
\tikzstyle{red}=[-, draw=red, line width=0.8pt]
\tikzstyle{blue thick}=[-, tikzit draw=blue, draw=blue, line width=1.7pt]
\tikzstyle{thick}=[-, ultra thick]
\tikzstyle{dot}=[-, dotted]
\tikzstyle{fade}=[-, opacity=0.5]
\tikzstyle{yellow}=[-, draw={rgb,255: red,208; green,208; blue,42}, line width=1.5]
\tikzstyle{arrow}=[->, line width=2em]
\tikzstyle{faded}=[-, opacity=0.5]
\tikzstyle{matching}=[-, draw={rgb,255: red,255; green,80; blue,80}, line width=5]
\tikzstyle{orange}=[-, draw={rgb,255: red,255; green,128; blue,0}]
\tikzstyle{red arrow}=[draw=red, ->]
\tikzstyle{double edge}=[-, double, double distance=0.8 pt, draw=blue, line width=1.3, opacity=1]
\tikzstyle{single edge}=[-, opacity=1, draw=blue, line width=1.3]

\tikzstyle{orange}=[-, draw={rgb,255: red,255; green,128; blue,0}]
\tikzstyle{dimer}=[-, draw={rgb,255: red,255; green,80; blue,80}, line width=2.5]
\tikzstyle{arrow}=[->]
\tikzstyle{grayfilled}=[-, fill={rgb,255: red,128; green,128; blue,128}, draw=blue, line width=2.5, opacity=0.5]

%% file: macro.tex
\newtheorem{lemma}{Lemma}[section]
\newtheorem{theorem}[lemma]{Theorem}

\newtheorem{prop}[lemma]{Proposition}
\newtheorem{cor}[lemma]{Corollary}
\newtheorem{conj}[lemma]{Conjecture}

\newtheorem{question}[lemma]{Question}

\newtheorem{lem-def}[lemma]{Lemma/Definition}

\newtheorem{definition}[lemma]{Definition}
\newtheorem{example}[lemma]{Example}

\newtheorem{notation}[lemma]{Notation}

\newtheorem{Question}[lemma]{Question}
\newtheorem{remark}[lemma]{Remark}

\newtheorem{Miracle}[lemma]{Miracle}

\newcommand{\col}{\operatorname{col}}

\newcommand{\im}{\operatorname{im}}

\newcommand{\RNum}[1]{\uppercase\expandafter{\romannumeral #1\relax}}

\DeclareTextFontCommand{\emph}{\color{blue}\em}

\newcommand{\bx}{\mathbf{x}}

%% file: Figure/1-x.tex
\begin{tikzpicture}[scale=0.7]
	\begin{pgfonlayer}{nodelayer}
		\node [style=vert] (0) at (-1, 0) {};
		\node [style=vert] (1) at (-1, 1) {};
		\node [style=vert] (2) at (-1, 2) {};
        \node [style=vert] (3) at (-1, 3) {};
		\node [style=none] (5) at (-1, 3.25) {.};
		\node [style=none] (7) at (-1, 3.5) {.};
		\node [style=none] (8) at (-1, 3.75) {.};
	\end{pgfonlayer}
	\begin{pgfonlayer}{edgelayer}
		\draw (3) to (2);
		\draw (2) to (1);
		\draw (1) to (0);
	\end{pgfonlayer}
\end{tikzpicture}

%% file: Figure/nraytree.tex
\begin{tikzpicture}[scale=0.7]
	\begin{pgfonlayer}{nodelayer}
		\node [style=vert] (0) at (0, 0) {};
		\node [style=vert] (1) at (-1, 1) {};
		\node [style=vert] (2) at (1, 1) {};
		\node [style=vert] (3) at (-1.5, 2) {};
		\node [style=vert] (4) at (-0.5, 2) {};
		\node [style=vert] (5) at (0.5, 2) {};
		\node [style=vert] (6) at (1.5, 2) {};
		\node [style=vert] (7) at (-1.75, 3) {};
		\node [style=vert] (8) at (-1.25, 3) {};
		\node [style=vert] (9) at (-0.75, 3) {};
		\node [style=vert] (10) at (-0.25, 3) {};
		\node [style=vert] (11) at (0.25, 3) {};
		\node [style=vert] (12) at (0.75, 3) {};
		\node [style=vert] (13) at (1.25, 3) {};
		\node [style=vert] (14) at (1.75, 3) {};
		\node [style=none] (15) at (0, 3.25) {};
		\node [style=none] (16) at (0, 3.5) {};
		\node [style=none] (17) at (0, 3.75) {};
		\node [style=none] (18) at (0, 3.25) {.};
		\node [style=none] (19) at (0, 3.5) {.};
		\node [style=none] (20) at (0, 3.75) {.};
	\end{pgfonlayer}
	\begin{pgfonlayer}{edgelayer}
		\draw (1) to (0);
		\draw (0) to (2);
		\draw (2) to (6);
		\draw (5) to (2);
		\draw (4) to (1);
		\draw (3) to (1);
		\draw (3) to (7);
		\draw (8) to (3);
		\draw (4) to (10);
		\draw (9) to (4);
		\draw (11) to (5);
		\draw (5) to (12);
		\draw (13) to (6);
		\draw (6) to (14);
	\end{pgfonlayer}
\end{tikzpicture}

%% file: Figure/stern.tex
\begin{tikzpicture}[scale=0.7]
	\begin{pgfonlayer}{nodelayer}
		\node [style=vert] (0) at (0, -1) {};
		\node [style=vert] (1) at (-1, 0) {};
		\node [style=vert] (2) at (0, 0) {};
		\node [style=vert] (3) at (1, 0) {};
		\node [style=vert] (4) at (-1.5, 1) {};
		\node [style=vert] (5) at (-1, 1) {};
		\node [style=vert] (6) at (-0.5, 1) {};
		\node [style=vert] (7) at (0, 1) {};
		\node [style=vert] (8) at (0.5, 1) {};
		\node [style=vert] (9) at (1, 1) {};
		\node [style=vert] (10) at (1.5, 1) {};
		\node [style=none] (14) at (0, 2.25) {.};
		\node [style=none] (15) at (0, 2.5) {.};
		\node [style=none] (16) at (0, 2.75) {.};
		\node [style=vert] (17) at (-1.25, 2) {};
		\node [style=vert] (18) at (-1.75, 2) {};
		\node [style=vert] (19) at (-1.5, 2) {};
		\node [style=vert] (20) at (-1, 2) {};
		\node [style=vert] (21) at (-0.75, 2) {};
		\node [style=vert] (22) at (-0.5, 2) {};
		\node [style=vert] (23) at (-0.25, 2) {};
		\node [style=vert] (24) at (0, 2) {};
		\node [style=vert] (25) at (0.25, 2) {};
		\node [style=vert] (26) at (0.5, 2) {};
		\node [style=vert] (27) at (0.75, 2) {};
		\node [style=vert] (28) at (1, 2) {};
		\node [style=vert] (29) at (1.25, 2) {};
		\node [style=vert] (30) at (1.5, 2) {};
		\node [style=vert] (31) at (1.75, 2) {};
	\end{pgfonlayer}
	\begin{pgfonlayer}{edgelayer}
		\draw (0) to (1);
		\draw (0) to (2);
		\draw (0) to (3);
		\draw (3) to (10);
		\draw (3) to (9);
		\draw (3) to (8);
		\draw (8) to (2);
		\draw (2) to (7);
		\draw (2) to (6);
		\draw (6) to (1);
		\draw (1) to (5);
		\draw (1) to (4);
		\draw (18) to (4);
		\draw (19) to (4);
		\draw (17) to (4);
		\draw (17) to (5);
		\draw (20) to (5);
		\draw (21) to (5);
		\draw (21) to (6);
		\draw (22) to (6);
		\draw (23) to (6);
		\draw (23) to (7);
		\draw (24) to (7);
		\draw (25) to (7);
		\draw (25) to (8);
		\draw (26) to (8);
		\draw (27) to (8);
		\draw (27) to (9);
		\draw (28) to (9);
		\draw (29) to (9);
		\draw (29) to (10);
		\draw (30) to (10);
		\draw (31) to (10);
	\end{pgfonlayer}
\end{tikzpicture}

%% file: Figure/bowtie.tex
\begin{tikzpicture}[scale=0.7]
	\begin{pgfonlayer}{nodelayer}
		\node [style=vert] (0) at (0, -1) {};
		\node [style=vert] (1) at (-1, 0) {};
		\node [style=vert] (2) at (1, 0) {};
		\node [style=vert] (3) at (-1, 1) {};
		\node [style=vert] (4) at (1, 1) {};
		\node [style=vert] (5) at (-1, 2) {};
		\node [style=vert] (6) at (1, 2) {};
		\node [style=vert] (7) at (1, 0) {};
		\node [style=none] (8) at (0, 2.25) {};
		\node [style=none] (9) at (0, 2.5) {};
		\node [style=none] (10) at (0, 2.75) {};
		\node [style=none] (11) at (0, 2.25) {.};
		\node [style=none] (12) at (0, 2.5) {.};
		\node [style=none] (13) at (0, 2.75) {.};
	\end{pgfonlayer}
	\begin{pgfonlayer}{edgelayer}
		\draw (0) to (7);
		\draw (0) to (1);
		\draw (1) to (4);
		\draw (3) to (7);
		\draw (1) to (3);
		\draw (7) to (4);
		\draw (3) to (5);
		\draw (6) to (4);
		\draw (6) to (3);
		\draw (5) to (4);
	\end{pgfonlayer}
\end{tikzpicture}

%% file: Figure/feigenerate.tex
\begin{tikzpicture}
	\begin{pgfonlayer}{nodelayer}
		\node [style=vert] (0) at (-5, 0) {};
		\node [style=vert] (1) at (-6, 1) {};
		\node [style=vert] (2) at (-5, 1) {};
		\node [style=vert] (3) at (-4, 1) {};
		\node [style=none] (4) at (-5.25, 1.25) {};
		\node [style=none] (5) at (-3.75, 1.25) {};
		\node [style=none] (6) at (-5.25, 0.75) {};
		\node [style=none] (7) at (-3.75, 0.75) {};
		\node [style=vert] (8) at (-2, 0) {};
		\node [style=vert] (9) at (-1, 0) {};
		\node [style=vert] (10) at (-3, 1) {};
		\node [style=vert] (11) at (-2, 1) {};
		\node [style=vert] (12) at (-1, 1) {};
		\node [style=vert] (13) at (0, 1) {};
		\node [style=none] (14) at (-2.25, 0.25) {};
		\node [style=none] (15) at (-0.75, 0.25) {};
		\node [style=none] (16) at (-2.25, -0.25) {};
		\node [style=none] (17) at (-0.75, -0.25) {};
		\node [style=none] (18) at (-2.25, 0.75) {};
		\node [style=none] (19) at (-0.75, 0.75) {};
		\node [style=none] (20) at (-2.25, 1.25) {};
		\node [style=none] (21) at (-0.75, 1.25) {};
	\end{pgfonlayer}
	\begin{pgfonlayer}{edgelayer}
		\draw (1) to (0);
		\draw (2) to (0);
		\draw (0) to (3);
		\draw (10) to (8);
		\draw (8) to (11);
		\draw (8) to (12);
		\draw (11) to (9);
		\draw (9) to (12);
		\draw (9) to (13);
		\draw [style=blue] (4.center) to (5.center);
		\draw [style=blue] (5.center) to (7.center);
		\draw [style=blue] (7.center) to (6.center);
		\draw [style=blue] (6.center) to (4.center);
		\draw [style=blue] (20.center) to (21.center);
		\draw [style=blue] (21.center) to (19.center);
		\draw [style=blue] (19.center) to (18.center);
		\draw [style=blue] (18.center) to (20.center);
		\draw [style=blue] (14.center) to (15.center);
		\draw [style=blue] (15.center) to (17.center);
		\draw [style=blue] (17.center) to (16.center);
		\draw [style=blue] (16.center) to (14.center);
	\end{pgfonlayer}
\end{tikzpicture}

%% file: Figure/fei.tex
\begin{tikzpicture}[scale=0.7]
	\begin{pgfonlayer}{nodelayer}
		\node [style=vert] (0) at (0, -2) {};
		\node [style=vert] (1) at (-2, -1) {};
		\node [style=vert] (2) at (2, -1) {};
		\node [style=vert] (3) at (3, -1) {};
		\node [style=vert] (4) at (-3.5, 0) {};
		\node [style=vert] (5) at (-2, 0) {};
		\node [style=vert] (6) at (-1.25, 0) {};
		\node [style=vert] (7) at (1, 0) {};
		\node [style=vert] (8) at (2, 0) {};
		\node [style=vert] (9) at (3, 0) {};
		\node [style=vert] (10) at (5, 0) {};
		\node [style=vert] (11) at (-5, 1) {};
		\node [style=vert] (12) at (-4, 1) {};
		\node [style=vert] (13) at (-3.25, 1) {};
		\node [style=vert] (14) at (-2.5, 1) {};
		\node [style=vert] (15) at (-2, 1) {};
		\node [style=vert] (16) at (-1.25, 1) {};
		\node [style=vert] (17) at (-0.75, 1) {};
		\node [style=vert] (18) at (-0.25, 1) {};
		\node [style=vert] (19) at (1.5, 1) {};
		\node [style=vert] (20) at (2, 1) {};
		\node [style=vert] (21) at (3, 1) {};
		\node [style=vert] (22) at (3.5, 1) {};
		\node [style=vert] (26) at (4, 1) {};
		\node [style=vert] (27) at (5, 1) {};
		\node [style=vert] (28) at (6, 1) {};
		\node [style=vert] (29) at (0.25, 1) {};
		\node [style=vert] (30) at (0.75, 1) {};
		\node [style=none] (31) at (1.75, -0.75) {};
		\node [style=none] (32) at (3.25, -0.75) {};
		\node [style=none] (33) at (3.25, -1.25) {};
		\node [style=none] (34) at (1.75, -1.25) {};
		\node [style=none] (35) at (1.75, 0.25) {};
		\node [style=none] (36) at (3.25, 0.25) {};
		\node [style=none] (37) at (3.25, -0.25) {};
		\node [style=none] (38) at (1.75, -0.25) {};
		\node [style=none] (39) at (1.75, 1.25) {};
		\node [style=none] (40) at (3.25, 1.25) {};
		\node [style=none] (41) at (3.25, 0.75) {};
		\node [style=none] (42) at (1.75, 0.75) {};
		\node [style=none] (43) at (4.75, 1.25) {};
		\node [style=none] (44) at (6.25, 1.25) {};
		\node [style=none] (45) at (6.25, 0.75) {};
		\node [style=none] (46) at (4.75, 0.75) {};
		\node [style=none] (47) at (-2.25, -0.25) {};
		\node [style=none] (48) at (-2.25, 0.25) {};
		\node [style=none] (49) at (-1, 0.25) {};
		\node [style=none] (50) at (-1, -0.25) {};
		\node [style=none] (51) at (-2.25, 0.75) {};
		\node [style=none] (52) at (-2.25, 1.25) {};
		\node [style=none] (53) at (-1, 1.25) {};
		\node [style=none] (54) at (-1, 0.75) {};
		\node [style=none] (55) at (0, 0.75) {};
		\node [style=none] (56) at (0, 1.25) {};
		\node [style=none] (57) at (1, 1.25) {};
		\node [style=none] (58) at (1, 0.75) {};
		\node [style=none] (59) at (-4.25, 0.75) {};
		\node [style=none] (60) at (-4.25, 1.25) {};
		\node [style=none] (61) at (-3, 1.25) {};
		\node [style=none] (62) at (-3, 0.75) {};
		\node [style=none] (63) at (-1, -0.25) {};
	\end{pgfonlayer}
	\begin{pgfonlayer}{edgelayer}
		\draw (1) to (0);
		\draw (2) to (0);
		\draw (0) to (3);
		\draw (4) to (1);
		\draw (1) to (6);
		\draw (5) to (1);
		\draw (7) to (2);
		\draw (8) to (2);
		\draw (9) to (2);
		\draw (8) to (3);
		\draw (9) to (3);
		\draw (3) to (10);
		\draw (11) to (4);
		\draw (4) to (12);
		\draw (4) to (13);
		\draw (14) to (5);
		\draw (15) to (5);
		\draw (15) to (6);
		\draw (16) to (6);
		\draw (16) to (5);
		\draw (6) to (17);
		\draw (18) to (7);
		\draw (10) to (26);
		\draw (10) to (27);
		\draw (10) to (28);
		\draw (19) to (8);
		\draw (8) to (20);
		\draw (8) to (21);
		\draw (9) to (20);
		\draw (9) to (21);
		\draw (9) to (22);
		\draw (29) to (7);
		\draw (30) to (7);
		\draw [style=blue] (60.center) to (61.center);
		\draw [style=blue] (61.center) to (62.center);
		\draw [style=blue] (62.center) to (59.center);
		\draw [style=blue] (59.center) to (60.center);
		\draw [style=blue] (52.center) to (53.center);
		\draw [style=blue] (53.center) to (54.center);
		\draw [style=blue] (54.center) to (51.center);
		\draw [style=blue] (51.center) to (52.center);
		\draw [style=blue] (48.center) to (49.center);
		\draw [style=blue] (49.center) to (63.center);
		\draw [style=blue] (63.center) to (47.center);
		\draw [style=blue] (47.center) to (48.center);
		\draw [style=blue] (56.center) to (57.center);
		\draw [style=blue] (57.center) to (58.center);
		\draw [style=blue] (58.center) to (55.center);
		\draw [style=blue] (55.center) to (56.center);
		\draw [style=blue] (39.center) to (40.center);
		\draw [style=blue] (40.center) to (41.center);
		\draw [style=blue] (41.center) to (42.center);
		\draw [style=blue] (42.center) to (39.center);
		\draw [style=blue] (35.center) to (36.center);
		\draw [style=blue] (36.center) to (37.center);
		\draw [style=blue] (37.center) to (38.center);
		\draw [style=blue] (38.center) to (35.center);
		\draw [style=blue] (31.center) to (32.center);
		\draw [style=blue] (32.center) to (33.center);
		\draw [style=blue] (33.center) to (34.center);
		\draw [style=blue] (34.center) to (31.center);
		\draw [style=blue] (43.center) to (44.center);
		\draw [style=blue] (44.center) to (45.center);
		\draw [style=blue] (45.center) to (46.center);
		\draw [style=blue] (46.center) to (43.center);
	\end{pgfonlayer}
\end{tikzpicture}

%% file: Figure/semiupho.tex
\begin{tikzpicture}[scale=0.6]
	\begin{pgfonlayer}{nodelayer}
		\node [style=vert] (0) at (1, -2) {};
		\node [style=vert] (1) at (-4, -1) {};
		\node [style=vert] (3) at (6, -1) {};
		\node [style=vert] (4) at (-6, 0) {};
		\node [style=vert] (5) at (-4, 0) {};
		\node [style=vert] (6) at (-2, 0) {};
		\node [style=vert] (10) at (4, 0) {};
		\node [style=vert] (12) at (-6.75, 1) {};
		\node [style=vert] (13) at (-6, 1) {};
		\node [style=vert] (14) at (-5.25, 1) {};
		\node [style=vert] (15) at (-4.75, 1) {};
		\node [style=vert] (16) at (-4, 1) {};
		\node [style=vert] (17) at (-3.25, 1) {};
		\node [style=vert] (18) at (-2.75, 1) {};
		\node [style=vert] (30) at (-7, 2) {};
		\node [style=vert] (31) at (-6.75, 2) {};
		\node [style=vert] (32) at (-6.5, 2) {};
		\node [style=vert] (33) at (-6.25, 2) {};
		\node [style=vert] (34) at (-6, 2) {};
		\node [style=vert] (35) at (-5.75, 2) {};
		\node [style=vert] (36) at (-5.5, 2) {};
		\node [style=vert] (37) at (-5, 2) {};
		\node [style=vert] (38) at (-4.75, 2) {};
		\node [style=vert] (40) at (-4.5, 2) {};
		\node [style=vert] (41) at (-4.25, 2) {};
		\node [style=vert] (42) at (-4, 2) {};
		\node [style=vert] (43) at (-3.75, 2) {};
		\node [style=vert-red] (51) at (1, -1) {};
		\node [style=vert-red] (52) at (-1, 0) {};
		\node [style=vert-red] (53) at (1, 0) {};
		\node [style=vert-red] (54) at (3, 0) {};
		\node [style=vert-red] (55) at (-1.75, 1) {};
		\node [style=vert-red] (56) at (-1, 1) {};
		\node [style=vert-red] (57) at (-0.25, 1) {};
		\node [style=vert-red] (58) at (0.25, 1) {};
		\node [style=vert-red] (59) at (1, 1) {};
		\node [style=vert-red] (60) at (1.75, 1) {};
		\node [style=vert-red] (61) at (-2, 2) {};
		\node [style=vert-red] (62) at (-1.75, 2) {};
		\node [style=vert-red] (63) at (-1.5, 2) {};
		\node [style=vert-red] (64) at (-1.25, 2) {};
		\node [style=vert-red] (65) at (-1, 2) {};
		\node [style=vert-red] (66) at (-0.75, 2) {};
		\node [style=vert-red] (67) at (-0.5, 2) {};
		\node [style=vert-red] (68) at (0, 2) {};
		\node [style=vert-red] (69) at (0.25, 2) {};
		\node [style=vert-red] (70) at (0.5, 2) {};
		\node [style=vert-red] (71) at (0.75, 2) {};
		\node [style=vert-red] (72) at (1, 2) {};
		\node [style=vert-red] (73) at (1.25, 2) {};
		\node [style=none] (74) at (9, 2) {};
	\end{pgfonlayer}
	\begin{pgfonlayer}{edgelayer}
		\draw (1) to (0);
		\draw (0) to (3);
		\draw (51) to (0);
		\draw [style=red] (52) to (51);
		\draw [style=red] (53) to (51);
		\draw [style=red] (51) to (54);
		\draw (4) to (1);
		\draw (5) to (1);
		\draw (1) to (6);
		\draw (12) to (4);
		\draw (13) to (4);
		\draw (14) to (4);
		\draw (15) to (5);
		\draw (16) to (5);
		\draw (17) to (5);
		\draw (18) to (6);
		\draw (10) to (3);
		\draw [style=red] (55) to (52);
		\draw [style=red] (56) to (52);
		\draw [style=red] (57) to (52);
		\draw [style=red] (58) to (53);
		\draw [style=red] (59) to (53);
		\draw [style=red] (60) to (53);
		\draw (30) to (12);
		\draw (31) to (12);
		\draw (32) to (12);
		\draw (33) to (13);
		\draw (34) to (13);
		\draw (35) to (13);
		\draw (36) to (14);
		\draw (37) to (15);
		\draw (38) to (15);
		\draw (40) to (15);
		\draw (41) to (16);
		\draw (42) to (16);
		\draw (43) to (16);
		\draw [style=red] (61) to (55);
		\draw [style=red] (62) to (55);
		\draw [style=red] (63) to (55);
		\draw [style=red] (64) to (56);
		\draw [style=red] (65) to (56);
		\draw [style=red] (66) to (56);
		\draw [style=red] (67) to (57);
		\draw [style=red] (68) to (58);
		\draw [style=red] (69) to (58);
		\draw [style=red] (70) to (58);
		\draw [style=red] (71) to (59);
		\draw [style=red] (72) to (59);
		\draw [style=red] (73) to (59);
	\end{pgfonlayer}
\end{tikzpicture}

%% file: Figure/colorex1.tex
\begin{tikzpicture}[scale=0.5]
	\begin{pgfonlayer}{nodelayer}
		\node [style=vert] (0) at (0, -1) {};
		\node [style=vert] (1) at (-0.75, 0) {};
		\node [style=vert] (2) at (0.75, 0) {};
		\node [style=vert] (3) at (-1.5, 1) {};
		\node [style=vert] (4) at (1.5, 1) {};
		\node [style=vert] (5) at (-2.25, 2) {};
		\node [style=vert] (6) at (2.25, 2) {};
		\node [style=vert] (7) at (-3, 3) {};
		\node [style=vert] (8) at (3, 3) {};
		\node [style=vert] (9) at (0, 1) {};
		\node [style=vert] (10) at (-0.75, 2) {};
		\node [style=vert] (11) at (0.75, 2) {};
		\node [style=vert] (12) at (-1.5, 3) {};
		\node [style=vert] (13) at (1.5, 3) {};
		\node [style=vert] (14) at (0, 3) {};
		\node [style=vert] (15) at (7, -1) {};
		\node [style=vert] (16) at (6.25, 0) {};
		\node [style=vert] (17) at (7.75, 0) {};
		\node [style=vert] (18) at (5.5, 1) {};
		\node [style=vert] (19) at (8.5, 1) {};
		\node [style=vert] (20) at (4.75, 2) {};
		\node [style=vert] (21) at (9.25, 2) {};
		\node [style=vert] (22) at (4, 3) {};
		\node [style=vert] (23) at (10, 3) {};
		\node [style=vert] (24) at (7, 1) {};
		\node [style=vert] (25) at (6.25, 2) {};
		\node [style=vert] (26) at (7.75, 2) {};
		\node [style=vert] (27) at (5.5, 3) {};
		\node [style=vert] (28) at (8.5, 3) {};
		\node [style=vert] (29) at (7, 3) {};
		\node [style=none] (30) at (0, 3.5) {.};
		\node [style=none] (31) at (0, 3.75) {.};
		\node [style=none] (32) at (0, 4) {.};
		\node [style=none] (33) at (7, 3.5) {.};
		\node [style=none] (34) at (7, 3.75) {.};
		\node [style=none] (35) at (7, 4) {.};
	\end{pgfonlayer}
	\begin{pgfonlayer}{edgelayer}
		\draw [style=red] (0) to (1);
		\draw [style=red] (9) to (1);
		\draw [style=red] (9) to (10);
		\draw [style=red] (10) to (14);
		\draw [style=red] (12) to (5);
		\draw [style=red] (5) to (3);
		\draw [style=red] (13) to (11);
		\draw [style=red] (11) to (4);
		\draw [style=red] (4) to (2);
		\draw [style=red] (8) to (6);
		\draw [style=blue] (2) to (0);
		\draw [style=blue] (9) to (2);
		\draw [style=blue] (11) to (9);
		\draw [style=blue] (14) to (11);
		\draw [style=blue] (13) to (6);
		\draw [style=blue] (6) to (4);
		\draw [style=blue] (12) to (10);
		\draw [style=blue] (10) to (3);
		\draw [style=blue] (3) to (1);
		\draw [style=blue] (7) to (5);
		\draw [style=red] (15) to (16);
		\draw [style=red] (24) to (25);
		\draw [style=red] (20) to (18);
		\draw [style=red] (26) to (19);
		\draw [style=blue] (17) to (15);
		\draw [style=blue] (26) to (24);
		\draw [style=blue] (21) to (19);
		\draw [style=blue] (25) to (18);
		\draw [style=red] (18) to (16);
		\draw [style=red] (22) to (20);
		\draw [style=red] (27) to (25);
		\draw [style=red] (24) to (17);
		\draw [style=red] (26) to (29);
		\draw [style=red] (28) to (21);
		\draw [style=blue] (19) to (17);
		\draw [style=blue] (23) to (21);
		\draw [style=blue] (24) to (16);
		\draw [style=blue] (28) to (26);
		\draw [style=blue] (29) to (25);
		\draw [style=blue] (27) to (20);
	\end{pgfonlayer}
\end{tikzpicture}

%% file: Figure/semiupho_and_tree.tex
\begin{tikzpicture}[scale=0.5]
	\begin{pgfonlayer}{nodelayer}
		\node [style=vert] (0) at (0, -1) {};
		\node [style=vert] (1) at (-0.75, 0) {};
		\node [style=vert] (2) at (0.75, 0) {};
		\node [style=vert] (3) at (-1.5, 1) {};
		\node [style=vert] (4) at (1.5, 1) {};
		\node [style=vert] (5) at (-2.25, 2) {};
		\node [style=vert] (6) at (2.25, 2) {};
		\node [style=vert] (7) at (-3, 3) {};
		\node [style=vert] (8) at (3, 3) {};
		\node [style=vert] (9) at (0, 1) {};
		\node [style=vert] (10) at (-0.75, 2) {};
		\node [style=vert] (11) at (0.75, 2) {};
		\node [style=vert] (12) at (-1.5, 3) {};
		\node [style=vert] (13) at (1.5, 3) {};
		\node [style=vert] (14) at (0, 3) {};
		\node [style=vert] (15) at (7, -1) {};
		\node [style=vert] (16) at (6.25, 0) {};
		\node [style=vert] (17) at (7.75, 0) {};
		\node [style=vert] (18) at (5.5, 1) {};
		\node [style=vert] (19) at (8.5, 1) {};
		\node [style=vert] (20) at (4.75, 2) {};
		\node [style=vert] (21) at (9.25, 2) {};
		\node [style=vert] (22) at (4, 3) {};
		\node [style=vert] (23) at (10, 3) {};
		\node [style=vert] (24) at (7, 1) {};
		\node [style=vert] (25) at (6.25, 2) {};
		\node [style=vert] (26) at (7.75, 2) {};
		\node [style=vert] (27) at (5.5, 3) {};
		\node [style=vert] (28) at (8.5, 3) {};
		\node [style=vert] (29) at (7, 3) {};
		\node [style=none] (30) at (0, 3.5) {.};
		\node [style=none] (31) at (0, 3.75) {.};
		\node [style=none] (32) at (0, 4) {.};
		\node [style=none] (33) at (7, 3.5) {.};
		\node [style=none] (34) at (7, 3.75) {.};
		\node [style=none] (35) at (7, 4) {.};
	\end{pgfonlayer}
	\begin{pgfonlayer}{edgelayer}
		\draw [style=red] (0) to (1);
		\draw [style=blue] (9) to (1);
		\draw [style=red] (9) to (10);
		\draw [style=blue] (10) to (14);
		\draw [style=blue] (12) to (5);
		\draw [style=red] (5) to (3);
		\draw [style=blue] (13) to (11);
		\draw [style=red] (11) to (4);
		\draw [style=blue] (4) to (2);
		\draw [style=blue] (8) to (6);
		\draw [style=blue] (2) to (0);
		\draw [style=red] (9) to (2);
		\draw [style=blue] (11) to (9);
		\draw [style=red] (14) to (11);
		\draw [style=red] (13) to (6);
		\draw [style=blue] (6) to (4);
		\draw [style=red] (12) to (10);
		\draw [style=blue] (10) to (3);
		\draw [style=red] (3) to (1);
		\draw [style=red] (7) to (5);
		\draw [style=red] (15) to (16);
		\draw [style=red] (20) to (18);
		\draw [style=blue] (17) to (15);
		\draw [style=blue] (26) to (24);
		\draw [style=blue] (21) to (19);
		\draw [style=blue] (25) to (18);
		\draw [style=red] (18) to (16);
		\draw [style=red] (22) to (20);
		\draw [style=blue] (19) to (17);
		\draw [style=blue] (23) to (21);
		\draw [style=blue] (24) to (16);
		\draw [style=blue] (28) to (26);
		\draw [style=blue] (29) to (25);
		\draw [style=blue] (27) to (20);
	\end{pgfonlayer}
\end{tikzpicture}

%% file: Figure/bigwhole.tex
\begin{tikzpicture}[scale=0.6]
	\begin{pgfonlayer}{nodelayer}
		\node [style=vert] (0) at (-1, 0) {};
		\node [style=vert] (1) at (-1, 1) {};
		\node [style=vert] (2) at (-1, 2) {};
		\node [style=vert] (3) at (-1, 3) {};
		\node [style=vert] (4) at (-1, 2) {};
		\node [style=label] (5) at (-1, 3.25) {};
		\node [style=label] (6) at (-1, 3.25) {.};
		\node [style=label] (7) at (-1, 3.5) {};
		\node [style=label] (8) at (-1, 3.75) {};
		\node [style=label] (9) at (-1, 3.5) {.};
		\node [style=label] (10) at (-1, 3.75) {.};
		\node [style=vert] (11) at (0.75, 0) {};
		\node [style=vert] (12) at (1.25, 1) {};
		\node [style=vert] (13) at (0, 2) {};
		\node [style=vert] (14) at (0.5, 2) {};
		\node [style=vert] (15) at (1, 2) {};
		\node [style=vert] (16) at (0.25, 1) {};
		\node [style=vert] (17) at (10.25, -1) {};
		\node [style=vert] (18) at (8.25, 0) {};
		\node [style=vert] (19) at (6.25, 1) {};
		\node [style=vert] (20) at (4.25, 2) {};
		\node [style=vert] (21) at (2.25, 3) {};
		\node [style=vert] (22) at (9.75, 0) {};
		\node [style=vert] (23) at (10.75, 0) {};
		\node [style=vert] (24) at (9.5, 1) {};
		\node [style=vert] (25) at (10, 1) {};
		\node [style=vert] (26) at (10.5, 1) {};
		\node [style=vert] (27) at (7.75, 1) {};
		\node [style=vert] (28) at (8.75, 1) {};
		\node [style=vert] (29) at (7.5, 2) {};
		\node [style=vert] (30) at (8, 2) {};
		\node [style=vert] (31) at (8.5, 2) {};
		\node [style=vert] (32) at (5.75, 2) {};
		\node [style=vert] (33) at (6.75, 2) {};
		\node [style=vert] (34) at (5.5, 3) {};
		\node [style=vert] (35) at (6, 3) {};
		\node [style=vert] (36) at (6.5, 3) {};
		\node [style=vert] (37) at (3.75, 3) {};
		\node [style=vert] (38) at (4.75, 3) {};
		\node [style=none] (44) at (11.75, -1) {1};
		\node [style=none] (45) at (11.75, 0) {3};
		\node [style=none] (46) at (11.75, 1) {6};
		\node [style=none] (47) at (11.75, 2) {6};
		\node [style=none] (48) at (11.75, 3) {6};
		\node [style=none] (52) at (11.75, 3.5) {.};
		\node [style=none] (53) at (11.75, 3.75) {.};
		\node [style=none] (54) at (11.75, 4) {.};
		\node [style=vert] (55) at (21.25, -1) {};
		\node [style=vert] (56) at (19.25, 0) {};
		\node [style=vert] (57) at (17.25, 1) {};
		\node [style=vert] (58) at (15.25, 2) {};
		\node [style=vert] (59) at (13.25, 3) {};
		\node [style=vert] (60) at (20.75, 0) {};
		\node [style=vert] (61) at (21.75, 0) {};
		\node [style=vert] (62) at (20.5, 1) {};
		\node [style=vert] (63) at (21, 1) {};
		\node [style=vert] (64) at (21.5, 1) {};
		\node [style=vert] (65) at (18.75, 1) {};
		\node [style=vert] (66) at (19.75, 1) {};
		\node [style=vert] (67) at (18.5, 2) {};
		\node [style=vert] (68) at (19, 2) {};
		\node [style=vert] (69) at (19.5, 2) {};
		\node [style=vert] (70) at (16.75, 2) {};
		\node [style=vert] (71) at (17.75, 2) {};
		\node [style=vert] (72) at (16.5, 3) {};
		\node [style=vert] (73) at (17, 3) {};
		\node [style=vert] (74) at (17.5, 3) {};
		\node [style=vert] (75) at (14.75, 3) {};
		\node [style=vert] (76) at (15.75, 3) {};
		\node [style=none] (82) at (22.75, -1) {1};
		\node [style=none] (83) at (22.75, 0) {3};
		\node [style=none] (84) at (22.75, 1) {6};
		\node [style=none] (85) at (22.75, 2) {6};
		\node [style=none] (86) at (22.75, 3) {6};
		\node [style=none] (90) at (22.75, 3.5) {.};
		\node [style=none] (91) at (22.75, 3.75) {.};
		\node [style=none] (92) at (22.75, 4) {.};
		\node [style=none] (93) at (-1, -1.5) {$\tilde{P}$};
		\node [style=none] (94) at (0.75, -1.5) {$\tilde{S}$};
		\node [style=none] (96) at (21.25, -1.5) {$\tilde{P}\rtimes_{\bx}\tilde{S}$};
	\end{pgfonlayer}
	\begin{pgfonlayer}{edgelayer}
		\draw (3) to (4);
		\draw (4) to (1);
		\draw (1) to (0);
		\draw [style=red] (13) to (16);
		\draw [style=red] (16) to (11);
		\draw [style=red] (15) to (12);
		\draw [style=blue] (14) to (16);
		\draw [style=blue] (12) to (11);
		\draw (21) to (20);
		\draw (20) to (19);
		\draw (19) to (18);
		\draw (18) to (17);
		\draw [style=red] (22) to (17);
		\draw [style=red] (24) to (22);
		\draw [style=red] (26) to (23);
		\draw [style=red] (29) to (27);
		\draw [style=red] (27) to (18);
		\draw [style=red] (31) to (28);
		\draw [style=red] (32) to (19);
		\draw [style=red] (34) to (32);
		\draw [style=red] (36) to (33);
		\draw [style=red] (37) to (20);
		\draw [style=blue] (38) to (20);
		\draw [style=blue] (33) to (19);
		\draw [style=blue] (35) to (32);
		\draw [style=blue] (28) to (18);
		\draw [style=blue] (30) to (27);
		\draw [style=blue] (23) to (17);
		\draw [style=blue] (25) to (22);
		\draw (59) to (58);
		\draw (58) to (57);
		\draw (57) to (56);
		\draw (56) to (55);
		\draw [style=red] (60) to (55);
		\draw [style=red] (62) to (60);
		\draw [style=red] (64) to (61);
		\draw [style=red] (67) to (65);
		\draw [style=red] (65) to (56);
		\draw [style=red] (69) to (66);
		\draw [style=red] (70) to (57);
		\draw [style=red] (72) to (70);
		\draw [style=red] (74) to (71);
		\draw [style=red] (75) to (58);
		\draw [style=blue] (76) to (58);
		\draw [style=blue] (71) to (57);
		\draw [style=blue] (73) to (70);
		\draw [style=blue] (66) to (56);
		\draw [style=blue] (68) to (65);
		\draw [style=blue] (61) to (55);
		\draw [style=blue] (63) to (60);
		\draw (57) to (60);
		\draw (57) to (61);
		\draw [style=blue] (66) to (61);
		\draw (58) to (65);
		\draw (58) to (66);
		\draw (58) to (62);
		\draw (58) to (63);
		\draw (58) to (64);
		\draw (59) to (70);
		\draw (59) to (71);
		\draw (59) to (67);
		\draw (59) to (68);
		\draw (59) to (69);
		\draw [style=red] (62) to (67);
		\draw [style=red] (64) to (67);
		\draw [style=red] (67) to (72);
		\draw [style=red] (69) to (72);
		\draw [style=blue] (63) to (71);
		\draw [style=blue] (66) to (71);
		\draw [style=blue] (68) to (62);
		\draw [style=blue] (64) to (68);
		\draw [style=blue] (76) to (71);
		\draw [style=blue] (68) to (76);
		\draw [style=blue] (73) to (67);
		\draw [style=blue] (69) to (73);
	\end{pgfonlayer}
\end{tikzpicture}